\documentclass[11pt,a4paper,twoside]{article}
\usepackage{amsmath,amsfonts,amssymb,amsthm,graphics}%,hyperref}
\usepackage[mathscr]{euscript}
\newtheorem{Theorem}[equation]{Theorem}
\newtheorem{Corollary}[equation]{Corollary}
\newtheorem{Proposition}[equation]{Proposition}
\newtheorem{Lemma}[equation]{Lemma}
\newtheorem{Remark}[equation]{Remark}

\newtheorem{Example}[equation]{Example}
\def\Section#1{\section{#1}\setcounter{equation}{0}}

\def\bdots{\mathinner{\mkern1mu\raise1pt\hbox{.}\mkern2mu\raise4pt\hbox{.}
           \mkern2mu\raise7pt\vbox{\kern7pt\hbox{.}}\mkern1mu}}

\def\Fx{F^\times}
\def\Ex{E^\times}

\def\oF{{\mathfrak o}_F}

\def\pF{{\mathfrak p}_F}
\def\kF{{k_F}}
\def\piF{{\varpi}_F}

\def\oE{{\mathfrak o}_E}

\def\pE{{\mathfrak p}_E}

\def\oEi{{\mathfrak o}_{E_i}}
\def\pEi{{\mathfrak p}_{E_i}}

\def\ov#1{\overline {#1}}

\def\fa{{\mathfrak a}}
\def\fb{{\mathfrak b}}

\def\vF{{\nu_F}}
\def\vE{{\nu_E}}

\def\vlambda{{\nu_\L}}
\def\vlambdai{{\nu_{\L_i}}}
\def\vlambdaone{{\nu_{\L_1}}}
\def\vlambdatwo{{\nu_{\L_2}}}
\def\vEone{{\nu_{E_1}}}
\def\vEtwo{{\nu_{E_2}}}
\def\vEi{{\nu_{E_i}}}

\def\End{{\hbox{\rm End}}}
\def\Aut{{\hbox{\rm Aut}}}
\def\Hom{{\hbox{\rm Hom}\,}}
\def\max{\hbox{\rm max}\,}
\def\min{\hbox{\rm min}\,}
\def\dim{\hbox{\rm dim}\,}
\def\ker{\hbox{\rm ker}\,}

\def\det{\hbox{\rm det}\,}

\def\Ind{{\hbox{\rm Ind}}}
\def\cInd{c{\hbox{\rm-Ind}}}

\def\char{\hbox{\rm char}\,}

\DeclareMathOperator{\Sp}{Sp}
\DeclareMathOperator{\GSp}{GSp}
\DeclareMathOperator{\GL}{GL}

\def\U{{\mathrm U}}

\def\tr{{\hbox{\rm tr}}}

\def\trEF{{\hbox{\rm tr}}_{E/F}}

\def\trEOF{{\hbox{\rm tr}}_{E_0/F}}
\def\trEiF{{\hbox{\rm tr}}_{E_i/F}}

\def\Clob{{\mathcal C}(\Lambda,0,\beta)}

\def\CC{{\mathcal C}}

\def\BZ{{\mathbb Z}}

\def\BC{{\mathbb C}}
\def\BN{{\mathbb N}}
\def\BL{{\mathbb L}}
\def\BP{{\mathbb P}}

\def\bok{{\mathbf k}}

\def\boI{{\mathbf 1}}

\def\bk{{\mathbf k}}

\def\la{\langle}
\def\ra{\rangle}

\def\a{{\alpha}}
\def\b{{\beta}}
\def\g{{\gamma}}
\def\d{{\delta}}
\def\e{{\varepsilon}}
\def\k{{\kappa}}
\def\l{{\lambda}}
\def\s{{\sigma}}
\def\t{{\tau}}
\def\th{{\theta}}
\def\L{{\Lambda}}
\def\Th{{\Theta}}

\def\tH{{\tilde H}}
\def\tJ{{\tilde J}}

\def\oG{{\bar G}}
\def\oH{{\bar H}}
\def\oJ{{\bar J}}

\def\bs{\boldsymbol}

\def\labelenumi{{\rm(\roman{enumi})}}
\def\theenumi{\roman{enumi}}

\begin{document}

\title{Genericity of supercuspidal representations of $p$-adic $\Sp_4$}

%\title{G{\'e}n{\'e}ricit{\'e} des repr{\'e}sentations supercuspidales
%de $\Sp_4$ $p$-adique}
%
\author{Corinne Blondel and Shaun Stevens%
\thanks{The second author would like to thank the University of Paris
7 for its hospitality in Spring 2005, during which time most of this
research was undertaken.} 
} 
\date{\today}
\maketitle
\begin{abstract}
We describe the supercuspidal representations of $\Sp_4(F)$, for
$F$ a non-archimedean local field of residual characteristic 
different from $2$, and determine which are generic. 
\end{abstract}
%
%\begin{resume}
%\end{resume}
%
%{\textit{Mathematics Subject Classification (2000):} 22E50}
%

%%%%%%%%%%%%%%%%%%%%%%%%%%%%%%%%%%%%%%%%%%%%%%%%%%%%%%%
\section*{Introduction}
%%%%%%%%%%%%%%%%%%%%%%%%%%%%%%%%%%%%%%%%%%%%%%%%%%%%%%%

Let $F$ be a locally compact non-archimedean local field, with ring of
integers $\oF$, maximal ideal $\pF$, and residue field $\kF$. Whereas
every irreducible supercuspidal representation of $\GL_N(F)$ 
is \emph{generic} -- i.e., has a Whittaker model -- this is no longer
true for classical groups over $F$. The existence of non-generic
supercuspidal representations of classical groups is significant and
bears many consequences; let us cite a few of them in local
representation theory (global consequences are heavy as well, see for
instance~\cite{HPS}): First, of course, the fact that the definition
of the $L$-function attached to a representation of a classical group
is available only for generic representations, thanks to the work of
Shahidi; in particular, a characterization through local factors of
the local Langlands correspondence for a classical group is not fully
available. Also, reducibility of parabolic induction is completely
understood for $\GL_N(F)$ while in classical groups results are
complete only in the case of generic inducing representations. 

The most celebrated example of a non-generic supercuspidal
representation is the representation $\Theta_{10}$ of $\Sp_4(F)$,
induced from the inflation to $\Sp_4(\oF)$ of the cuspidal unipotent
representation $\theta_{10}$ of $\Sp_4(\kF)$ constructed by Srinivasan
(se~\cite{HPS} again). To our knowledge, until recently, most known
non-generic supercuspidal representations were level zero, in
particular level zero representations induced from the inflation of a
cuspidal unipotent representation of the reductive quotient of a
maximal special parahoric subgroup. Our purpose in this paper is to
exhaust the non-generic supercuspidal representations of $\Sp_4(F)$ in
odd residual characteristic.

It is certainly no surprise, and part of the folklore in the subject,
that level zero supercuspidal representations of $\Sp_4(F)$ coming
from cuspidal representations of $\Sp_4(\kF)$ are generic if and only
if the corresponding cuspidal representation is, and that the only
non-generic cuspidal representation of $\Sp_4(\kF)$ is Srinivasan's
$\theta_{10}$. For the sake of completeness we include proofs of these
results. Actually we prove that generic level zero supercuspidal
representations of $\Sp_{2N}(F)$ are obtained by inducing the
inflation of a generic cuspidal representation of the reductive
quotient of a maximal parahoric subgroup of $\Sp_{2N}(F)$, on the
condition that this parahoric subgroup is special.

For positive level supercuspidal representations, the situation is not
as simple. Our point of view is to use the exhaustive construction
given by the second author in~\cite{S5}: those representations are
induced from a set of types, generalizing Bushnell-Kutzko types for
$\GL_N(F)$. We give necessary and sufficient conditions on those types
for the induced representation to be generic. We obtain surprisingly
many (at least with respect to our starting point) non-generic
representations.

Let us be more precise. Positive level supercuspidal representations
of $\Sp_4(F)$ fall into four categories, according to the nature of
the skew semi-simple stratum $[\L,n,0,\b]$ at the bottom level of the
construction (\S\ref{construction}): 

\begin{itemize}
\item[(I)] The first category starts with a skew simple stratum with
corresponding field extension $F[\b]$ of dimension $4$ over $F$. Here
non-generic representations are obtained only when $F[\b]/F$ is the
biquadratic extension, and when a binary condition involving $\b$ and
the symplectic form is fulfilled.
\item[(II)] The second category starts with a skew simple stratum with
corresponding field extension $F[\b]$ of dimension $2$ over
$F$. Non-generic representations are obtained whenever the
$F[\b]$-skew hermitian form attached to the symplectic form is
anisotropic (or equivalently, when the quotient group $J/J^1$ involved
in the construction, which is a reductive group over $k_F$, is
anisotropic).
\item[(III)] In the third category, the stratum is the orthogonal sum
of two two-dimensional skew simple strata $[\L_i,n_i,0,\b_i]$,
$i=1,2$. Non-genericity occurs only when $F[\b_1]$ is isomorphic to
$F[\b_2]$, and again when a binary condition involving $\b_1 \b_2$ and
the symplectic form is fulfilled. 
\item[(IV)] All representations in the fourth category -- when the
stratum is the orthogonal sum of a skew simple stratum and a null
stratum, both two-dimensional -- are non-generic.
\end{itemize}

\medskip

The main character in the proof is indeed a would-be character: To a
stratum as above is attached a function $\psi_\b$ on $\Sp_4(F)$ and
the crucial question is whether there exists a maximal unipotent
subgroup $U$ of $\Sp_4(F)$ on which $\psi_\b$ actually defines a
character (see \cite{BH}); this question is easily solved
in~\S\ref{S4}, where Proposition~\ref{Uderprop} lists the exact
conditions alluded to above. Whenever there is no such subgroup $U$,
we prove in~\S\ref{S6} that the corresponding representations are not
generic, using a criterion of non-genericity given
in~\S\ref{S1criterion} (this says essentially that if a cuspidal
representation $\cInd_J^{\Sp_4(F)}\l$ is generic, then there is a long root
subgroup on which $\lambda$ is trivial).

Now assume that there is a maximal unipotent subgroup $U$ on which
$\psi_\b$ is a character. A type attached to our stratum is a
representation  $\kappa \otimes \sigma$ of a compact open subgroup
$J$, where $\kappa$ is a suitable $\b$-extension attached to the
stratum and $\sigma$ is a cuspidal representation of some finite
reductive group $J/J^1$ attached to the stratum
(see~\S\ref{construction}). The fundamental step is
Theorem~\ref{kappainduced}, stating that the representation $\kappa$
contains the character $\psi_\b$ of $J \cap U$. This implies
genericity in cases~(I) and~(III), where $\sigma$ is just a character
of an anisotropic group. But for cases~(II) and~(IV) (in which, we
should add, $\psi_\b$ defines a {\it degenerate} character of $U$) we
have to deal with the component $\sigma$, with opposite effects. In
case~(II), the inflation of $\sigma$ contains the restriction to
$J\cap U$ of a character of $U$ and genericity follows. On the
contrary, in case~(IV), the inflation of $\sigma$ does not contain
the restriction to $J\cap U$ of a character of $U$ and the resulting
cuspidal representation is not generic.

\medskip

We have several remarks to add concerning these results. 
First, genericity for positive level supercuspidal representations of
$\Sp_4(F)$ only depends on the stratum itself, on the symplectic form,
and on the representation $\sigma$ (the ``level 0 part''). It does not
depend on the choice of a semi-simple character attached to the
stratum.

The second remark is that the proofs are quite technical, often on a
case-by-case basis. The present work can of course be regarded as a
first step towards understanding non-genericity in classical groups,
but even the case of $\Sp_{2N}(F)$ might not be just an easy
generalisation. In particular, the precise conditions for genericity 
are surprisingly complicated but it seems to us that they do not
admit a simple unified description, as for level zero
representations. For example, non-generic positive level supercuspidal
representations can be induced from either special or non-special
maximal compact subgroups of $\Sp_4(F)$, and likewise for generic
representations.

Finally, we have deliberately stuck to the construction of supercuspidal
representations of $\Sp_4(F)$ via types. Another very fruitful point
of view uses Howe's correspondence. Indeed, in a very recent work
(\cite{GT}), Gan and Takeda study the Langlands correspondence for
$\GSp_4(F)$ and obtain in the process a classification of non-generic
supercuspidal representations of $\GSp_4(F)$ in terms of Howe's
correspondence; they also announce a sequel dealing with $\Sp_4(F)$. A
dictionary between these two points of view would of course be very
interesting, especially if it can provide some insight about the way
non-generic supercuspidal representations of $\Sp_4(F)$ fit into
$L$-packets. For example, in case (I), the genericity of the
supercuspidal representation depends on the embedding of $\b$ in the
symplectic Lie algebra: up to the adjoint action of $\Sp_4(F)$, there
are two such embeddings, precisely one of
which gives rise  to non-generic representations. This suggests that
representations in these two sets might be paired to form $L$-packets. 
A closer
investigation of such phenomena would certainly deserve some effort.

%%%%%%%%%%%%%%%%%%%%%%%%%%%%%%%%%%%%%%%%%%%%%%%%%%%%%%%
\section*{Notation}
%%%%%%%%%%%%%%%%%%%%%%%%%%%%%%%%%%%%%%%%%%%%%%%%%%%%%%%

Let $F$ be a locally compact non-archimedean local field, with ring of
integers $\oF$, maximal ideal $\pF$, residue field $\kF$ and odd residual
characteristic $p=\char\kF$. On occasion $\varpi_F$ will denote a uniformizing element of $F$. 
Similar notations will be used for field extensions of $F$. 
We let $\nu_F$ denote the additive valuation of $F$, normalised 
so that $\nu_F(F^\times)= \mathbb Z$. We fix, once and for all, an additive
character $\psi_F$ of $F$ with conductor $\pF$.

Let $V$ be a $2N$-dimensional $F$-vector space, equipped with a
nondegenerate alternating form $h$. By a {\it symplectic basis\/} for
$V$, we mean an ordered basis $\{e_{-N},...,   e_{-1}, e_1 ,...,e_N\}$
such that,
for $1\le i,j\le N$,
$$
h(e_i,e_j)\ =\ h(e_{-i},e_{-j})\ =\ 0,\quad h(e_{-i},e_j)=\d_{ij},
$$
where $\d_{ij}$ is the Kronecker delta. 

Let $A=\End_F(V)$ and let $\ov{\phantom{a}}$ denote the adjoint
anti-involution on $A$ associated to $h$, so
$$
h(av,w)\ =\ h(v,\ov{a}w),\quad\hbox{ for }a\in A,\, v,w\in V.
$$
We will let $G$ be the corresponding symplectic group 
$G = \{ g \in \GL_F(V) / \ov{g} = g^{-1} \}$, or $G = \Sp_{2N}(F)$ 
whenever a symplectic basis is fixed. For most of the paper
 $N$ will be $2$. Similarly $\bar G$ will denote either 
 $\GL_F(V)$ or $\GL_{2N}(F)$. 
 
Skew semi-simple strata in $A$ are the  basic objects in what
follows. We recall briefly the essential notations attached to those
objects and refer to~\cite{S4} for definitions.  

Let $[\L,n,0,\b]$ be a skew semisimple stratum in $A$. Then $\L$ is a
self-dual lattice sequence  in $V$ and defines a decreasing filtration
$\{ \mathfrak a_i(\L), i \in \mathbb Z\}$  of $A$ by $\oF$-lattices
$\mathfrak a_i(\L) = \{ x \in A / \forall k \in \mathbb Z , \ x \L(k)
\subseteq \L(k+i) \}$. We put $\mathfrak a (\L) = \mathfrak a_0(\L)$,
a self-dual $\oF$-order in $A$. We will also need $P(\L) = \mathfrak a
(\L)^\times \cap G$ and $P_i(\L) = (1 +  \mathfrak a_i(\L)) \cap G$
for $i \ge 1$. Note that the lattice sequence $\L$ gives rise to a
valuation $\nu_\L$ on $V$ (\ref{SA.1}) and the filtration $\{
\mathfrak a_i(\L), i \in \mathbb Z\}$ gives rise to a valuation on
$A$, also denoted by $\nu_\L$.  

Next, $\beta$ is a skew element in $A$: $\ov{\beta} = - \beta$, and 
$n$ is a positive integer with $n= - \nu_\L (\beta)$. Furthermore 
the algebra $E = F[\beta]$ is a sum of fields $E=\oplus_{i=1}^l E_i$, 
corresponding to a decomposition $V = \perp_{i=1}^l V^i$ of $V$ as an
orthogonal direct sum, and accordingly of $\L$: $\L = \Sigma_{i=1}^l
\L^i$ with $\L^i(k)=\L(k)\cap V^i$, and of $\beta$: $\beta =
\Sigma_{i=1}^l \beta_i$. The centralizer $B$ of $\beta$ in $A$ is $B
= \oplus_{i=1}^l B_i$ where $B_i$ is the centralizer of $\beta_i$ in $
\End_F(V^i)$.  
 
Last, for $\b$ a skew element of $A$ we define $\psi_\b$ as the
following function on $G$: $\psi_\b (x) = \psi\left( \tr (\b
(x-1))\right)$, $x \in G$. 

%%%%%%%%%%%%%%%%%%%%%%%%%%%%%%%%%%%%%%%%%%%%%%%%%%%%%%%
\Section{Generalities on genericity}\label{S1}
%%%%%%%%%%%%%%%%%%%%%%%%%%%%%%%%%%%%%%%%%%%%%%%%%%%%%%%

\subsection{Genericity}

The results of this subsection are valid in a much more general setting
than the remainder of this paper so we temporarily suspend our usual
notations. 

Let $G= \mathbf G(F)$ be the group of $F$-rational points of  a connected reductive
algebraic group $\mathbf G$ defined over $F$. Let $\mathbf S$ be a maximal $F$-split torus in $\mathbf G$
with $\mathbf G$-centralizer $\mathbf T$, let $\mathbf B$ be an $F$-parabolic 
 subgroup of $\mathbf G$ with Levi component $\mathbf T$, and let $\mathbf U$ be the unipotent radical of $\mathbf B$.

Let $\chi$ be a smooth (unitary) character of $U= \mathbf U(F)$. The torus $S = \mathbf S(F)$ acts
on the set of such characters by conjugation. We say that $\chi$ is
{\it nondegenerate\/} if its stabilizer for this action is just the
centre $Z$ of $G$.

\begin{Example}\rm Let $G=\Sp_4(F)$, which we write as a group of 
matrices with respect to some symplectic basis, let $T$ be the torus
of diagonal matrices, and let $U$ be the subgroup of upper triangular
unipotent matrices in $G$. Any character $\chi$ of $U$ is given by
$$
\chi\begin{pmatrix} 1&u&x&y\\ 0&1&v&x\\ 0&0&1&-u\\ 0&0&0&1
\end{pmatrix}\ =\ \psi_F(au+bv),
$$
for some $a,b\in F$, and it is easy to see that $\chi$ is
nondegenerate if and only if $a,b$ are both non-zero. Moreover, there
are four orbits of nondegenerate characters of $U$, given by the class
of $b$ in $\Fx/(\Fx)^2$.
\end{Example}

Returning to a general connected reductive group $G$, 
let $\pi$ be a smooth irreducible representation of $G$. We say
that $\pi$ is {\it generic\/} if there exist $U= \mathbf U(F)$ as above 
 and a nondegenerate character $\chi$ of $U$ such that
$$
\Hom_G(\pi,\Ind_U^G\chi)\ne 0.
$$
Note that, since all such subgroups $U$ are conjugate in
$G$ we may choose to fix one. Moreover, we need only consider
nondegenerate characters $\chi$ up to $T$-conjugacy. A basic result
here is:

\begin{Theorem}[\cite{R} Theorem 3] Assume $\mathbf G$ is split over $F$.
Let $\pi$ be a smooth irreducible
representation of $G$ and let $\chi$ be a nondegenerate character of a
maximal unipotent subgroup $U$ of $G$. Then
$$
\dim_{\BC}\Hom_G(\pi,\Ind_U^G\chi)\le 1.
$$
\end{Theorem}

On the other hand, and again in a general $G$, 
when dealing with supercuspidal representations we may not bother about 
the nondegeneracy of the character in the definition of genericity, a fact
that will be useful in the sequel. Indeed: 

\begin{Lemma}\label{nondegenerate}  
Assume $\mathbf G$ is split over $F$.
Let $\pi$ be a smooth irreducible supercuspidal 
representation of $G$. Let $\mathbf U$ be 
a maximal connected unipotent subgroup of $\mathbf G$ and let  $\chi$ be a  
character of $U= \mathbf U(F)$ such that 
$$
\Hom_G(\pi,\Ind_U^G\chi)\ne 0.
$$
Then the character $\chi$ is nondegenerate. 
\end{Lemma}

\begin{proof}
Assume $\chi$ is degenerate and use the definition of nondegeneracy in 
\cite{BH3}, 1.2, as well as the corresponding notation: there is a simple root $\alpha$ 
such that the restriction of $\chi$ to $U_{(\alpha)}$ is trivial. 
The character $\chi$ is then trivial on the subgroup $\left< U_{(\alpha)}, U_{\text der}\right>$ 
of $U$. 

We claim that this subgroup contains the unipotent radical $N$ of a proper parabolic subgroup 
$P$ of $G$. The assumption $\Hom_G(\pi,\Ind_U^G\chi)\ne 0$  provides us, by Frobenius reciprocity,  
with a non-zero linear form $\lambda$ 
on the space $V$ of $\pi$ which satisfies 
$\lambda \circ \pi(u) = \chi (u) \lambda$ for any $u \in U$. 
Since $\lambda$ is in particular $N$-invariant, we get that the space $V_N$ of $N$-coinvariants is non-zero, 
which contradicts cuspidality. 

We now prove the claim. Let $\Delta$ be the set of simple roots as in \cite{BH3}. 
From \cite{Bo}, 21.11, the unipotent radical of the  standard $F$-parabolic subgroup 
$\mathbf P_I$ of $\mathbf G$ attached to the subset $I = \Delta - \{\alpha\}$ is
$\mathbf U_{\Psi(I)}$ where $\Psi(I)$ is the set of positive roots that can be written 
$\alpha + \beta$ with $\beta$ either $0$ or a positive root. Hence the elements in $\Psi(I)$ 
other than $\alpha$ are positive roots $\gamma$ of length at least $2$, and  $\mathbf U_{\Psi(I)}$ 
is directly spanned by $\mathbf U_{(\alpha)}$ and the $\mathbf U_{(\gamma)}$ for non-divisible 
such roots $\gamma$. 
From \cite{BH2}, Theorem 4.1, for any such $\gamma$ we have 
$\mathbf U_{(\gamma)}(F) \subset   U_{\text der}$ (recall the characteristic of $F$ is not $2$), 
hence $N =  \mathbf U_{\Psi(I)}(F)$ is contained in $\left< U_{(\alpha)}, U_{\text der}\right>$ 
as asserted.
\end{proof}

Now let $\pi$ be an irreducible supercuspidal representation of
$G$. We suppose, as is the case for all known supercuspidal
representations, that $\pi$ is irreducibly compact-induced from some open
compact mod centre subgroup of $G$. Then (the proof of) \cite{BH}
1.6 Proposition immediately gives us the following:

\begin{Proposition}\label{BHprop}
Let $K$ be an open compact mod centre subgroup of
$G$ and $\rho$ an irreducible representation of $K$ such that
$\pi=\cInd_K^G\rho$ is an irreducible, so supercuspidal,
representation of $G$. Then $\pi$ is generic if and only if there
exist a maximal connected unipotent subgroup $\mathbf U$ of $\mathbf G$ and a  
character $\chi$ of $U= \mathbf U(F)$ such that $\rho|K\cap U$ contains $\chi|K\cap
U$. Moreover, if this is the case then $
\Hom_G(\pi,\Ind_U^G\chi)\ne 0 
$ so $\chi$ is nondegenerate,  and 
the character $\chi|K\cap U$
occurs in $\rho|K\cap U$ with multiplicity $1$.
\end{Proposition}

\subsection{A criterion for non-genericity}\label{S1criterion}

Now we look more closely at the situation for symplectic groups so we
return to our usual notation: $G=\Sp_{2N}(F)$. In particular, using
Proposition \ref{BHprop} and a decomposition of $G$, we will obtain a
criterion to determine when an irreducible representation of $G$ is
{\it not\/} generic.

Let $T$ denote the standard (diagonal) maximal split torus of $G$,
let $U$ be the subgroup of all upper triangular unipotent matrices in
$G$, and put ${\cal B}=TU$, the Borel subgroup of all upper triangular
matrices in $G$. Let $\Phi=\Phi(G,T)$ be the root system and, for
$\g\in\Phi$, let $U_\g$ denote the corresponding root subgroup. Let $W$
denote the Weyl group $N_G(T)/T$; by abuse of notation, we will also
use $W$ for a set of representatives in the compact maximal subgroup
$K_0=\Sp_{2N}(\oF)$ of $G$.

We write $K_1$ for the pro-unipotent radical of $K_0$, so that
$K_0/K_1\simeq \Sp_{2N}(k_F)$. We note that ${\cal B}\cap K_0/{\cal B}\cap K_1$ is
the standard Borel subgroup (of upper triangular matrices) of
$K_0/K_1$, that $T\cap K_0/T\cap K_1$ is the diagonal torus, and that
$W$ is the Weyl group. 

Let $I_1$ denote the inverse image of the maximal unipotent subgroup
$U\cap K_0/U\cap K_1$ of $K_0/K_1$, that is, the pro-unipotent radical
of the standard Iwahori subgroup $I$ consisting of matrices which are
upper triangular modulo $\pF$. Then the Bruhat decomposition for
$K_0/K_1$ gives
$$
K_0\,/\,K_1 = ({\cal B}\cap K_0) W I_1\,/\,K_1.
$$
Since $K_1\subset I_1$, we obtain $K_0=({\cal B}\cap K_0)WI_1$. Finally, using
the Iwasawa decomposition $G={\cal B}K_0$ (since $K_0$ is a good maximal
compact subgroup of $G$), we obtain
\begin{equation} \label{G=BWI}
G = {\cal B}WI_1.
\end{equation}

Now we can use this decomposition to translate Proposition
\ref{BHprop} into a sufficient condition for non-genericity of
compactly-induced supercuspidal representations of $G$.

\begin{Proposition} \label{criterion}
Let $J$ be a compact open subgroup of $G$ and $\lambda$ an irreducible
representation of $J$ such that $\pi=\cInd_J^G\l$ is an irreducible
supercuspidal representation of $G$. Let $U$ be the subgroup of all
upper triangular unipotent matrices in $G$ and we also use the other
notations from above. Then
$\pi$ is generic if and only if there exist $w\in W$, $p\in I_1$
and a   character $\chi$ of $U$ such that\/ ${}^p\l$
contains the character $\chi^w$ of\/ ${}^pJ\cap U^w$.
In particular, 
\begin{quote}
if $\pi$ is generic then there are $p\in I_1$ and a
\emph{long} root $\g\in\Phi$ such that\/ ${}^p\l$ contains the trivial
character of\/ ${}^pJ\cap U_\g$.
\end{quote}
\end{Proposition}

We remark that, in our symplectic basis, the long roots correspond to
the entries on the anti-diagonal.

\begin{proof} Since all maximal unipotent subgroups of $G$ are
conjugate to $U$, Proposition~\ref{BHprop} implies that $\pi$ is
generic if and only if there exist $g\in G$ and a  
character $\chi$ of $U$ such that $\l$ contains the character $\chi^g$
of $J\cap U^g$. Now we use the decomposition \eqref{G=BWI} to write
$g=bwp$, with $b\in B$, $w\in W$ and $p\in I_1$. Since $U^b=U$ and
$\chi^b$ is another character of $U$, we can absorb the
$b$ and the result follows on conjugating by $p$.

The final assertion follows since the derived subgroup $U^w_{der}$
contains $U_\g$ for some long root $\g$.
\end{proof}

\setcounter{section}{1}

%%%%%%%%%%%%%%%%%%%%%%%%%%%%%%%%%%%%%%%%%%%%%%%%%%%%%%%
\Section{The  supercuspidal representations of
${\bf Sp}_{\bs 4}\bs(\bs F\bs)$}\label{S3}
%%%%%%%%%%%%%%%%%%%%%%%%%%%%%%%%%%%%%%%%%%%%%%%%%%%%%%%

What we seek is a complete list of which supercuspidal representations 
of $\Sp_4(F)$ are generic, which are not. 
So we start with a description of positive level supercuspidal 
representations of $\Sp_4(F)$ -- 
level zero supercuspidal representations are obtained by inducing from
 a maximal parahoric subgroup $\mathcal P$ 
the inflation of a cuspidal representation of the (finite) quotient 
of $\mathcal P$ by its pro-$p$-radical. 
We are then in a position to state our main theorem, identifying 
non-generic representations from our list. We end the section with a proof 
of the theorem for level zero representations. 
The proof for positive level occupies the remaining sections. 

\subsection{The positive level supercuspidal representations }\label{construction}

In this section we describe the construction of the positive level
supercuspidal representations of $\Sp_4(F)$. We refer to \cite{S5} for
more details and for proofs of the results stated here.

The construction begins with a skew semisimple stratum $[\L,n,0,\b]$
in $A$ such that $\mathfrak a (\L) \cap B$ is a {\it maximal} self-dual order normalized by $E^\times$ in $B$. There are essentially four cases here. In the first two, the
stratum is actually simple:
\begin{enumerate}
\def\labelenumi{{\rm(\Roman{enumi})}}
\def\theenumi{\Roman{enumi}}
\item \label{case4} ``maximal case'': $[\L,n,0,\b]$ is a skew simple
stratum and $E=F[\b]$ is an extension of $F$ of degree $4$.
\item \label{2then0} ``2 then 0 case'': $[\L,n,0,\b]$ is a skew simple
stratum, $E=F[\b]$ is an extension of $F$ of degree $2$, and
$\fa_0(\L)$ is maximal amongst (self-dual) $\oF$-orders in $A$
normalized by $\Ex$.
\end{enumerate}
Otherwise, we have a splitting $V=V^1\perp V^2$ of $[\L,n,0,\b]$ into
two $2$-dimensional $F$-vector spaces, and we write: $\L^i$ for the
lattice sequence in $V^i$ given by $\L^i(k)=\L(k)\cap V^i$, for
$k\in\BZ$; $\b_i=\boI^i\b\boI^i$, where $\boI^i$ is the projection
onto $V^i$ with kernel $V^{3-i}$; if $\b_i\ne 0$ then
$n_i=-\nu_{\L^i}(\b_i)$, otherwise $n_i=0$. Then $n=\max\{n_1,n_2\}$.
\begin{enumerate}
\def\labelenumi{{\rm(\Roman{enumi})}}
\def\theenumi{\Roman{enumi}}
\setcounter{enumi}{2}
\item \label{2plus2} ``2+2 case'': for $i=1,2$, $[\L^i,n_i,0,\b_i]$ is a
skew simple stratum and $E_i=F[\b_i]$ is an extension of $F$ of degree $2$.
\item \label{2plus0} ``2+0 case'': $[\L^1,n_1,0,\b_1]$ is a
skew simple stratum and $E_1=F[\b_1]$ is an extension of $F$ of degree
$2$; $\b_2=0$, so that in $V^2$ we have the null stratum
$[\L^2,0,0,0]$, and $\fa_0(\L^2)$ is maximal amongst (self-dual)
$\oF$-orders in $A^2=\End_F(V^2)$.
\end{enumerate}
We often think of \eqref{2plus0} as a degenerate case of
\eqref{2plus2} by thinking of a null stratum as a degenerate simple
stratum.

In each case, we have the subgroups $\oH^1=H^1(\b,\L)$,
$\oJ^1=J^1(\b,\L)$ and $\oJ=J(\b,\L)$ of $\oG$. We write
$H^1=\oH^1\cap G$, and similarly for the other groups. There is a
family $\CC(\b,\L)$ of rather special characters of $H^1$, called
{\it semisimple characters\/}; one of their properties is the fact 
that their restriction to $P_i(\L)$ for suitable $i$ is 
equal to $\psi_\beta$.    For each $\th\in\CC(\b,\L)$, there is a
unique irreducible representation $\eta$ of $J^1$ containing $\th$.
In each case, there is a ``suitable'' extension $\k$ of $\eta$ to a
representation of $J$, which we call a {\it $\beta$-extension\/} --
see below for more details of this step.

The extensions $E$, $E_i$ in each case come equipped with a
non-trivial galois involution, which we write $\ov{\phantom{a}}$ as
usual. We use the same notation for the induced involution on the
residue fields $k_E$, $k_i$; note that this involution may be
trivial. Then the quotient $J/J^1$ has one of the following forms:
\begin{enumerate}
\def\labelenumi{{\rm(\Roman{enumi})}}
\def\theenumi{\Roman{enumi}}
\item $k_E^1=\{x\in k_E:x\ov x=1\}$;
\item $U(1,1)(k_E/k_F)$ or $k_E^1\times k_E^1$ if $E/F$ is unramified; 
$SL_2(k_F)$ or $O_2(k_F)$ if $E/F$ is ramified;
\item $k_1^1\times k_2^1$;
\item $k_1^1\times SL_2(k_F)$.
\end{enumerate}
Let $\s$ be the inflation to $J$ of an irreducible cuspidal
representation of $J/J^1$. (Note that, in the case of $O_2(k_F)$ in
\eqref{2then0}, this just means any irreducible representation of the
(anisotropic) dihedral group $O_2(k_F)$.)
% -- in general, when $J/J^1$ is
% not connected (as an algebraic group over $k_F$) cuspidal means that
% the restriction to the connected component is a direct sum of
% (equivalently, contains) cuspidal representations.

Now we put $\l=\k\otimes\s$ and $\pi=\cInd_J^G\l$ is an irreducible
supercuspidal representation of $G$. All irreducible supercuspidal
representations of $G$ can be constructed in this way (though we
remark that often we cannot, as yet, tell when two such
representations are equivalent).

\medskip

Finally in this subsection, we recall briefly some properties of the
$\b$-extensions which we will need, especially in the cases
\eqref{2then0} and \eqref{2plus0} where their construction is somewhat
more involved. Indeed, in cases \eqref{case4} and \eqref{2plus2},
$J/J^1$ has no unipotent elements so there is never any problem here.

We define another skew semisimple stratum $[\L_m,n_m,0,\b]$ as follows:
\begin{itemize}
\item in cases \eqref{case4} and \eqref{2plus2}, we have $\L_m=\L$, $n_m=n$;
\item in case \eqref{2then0}, $\L_m$ is a self-dual $\oE$-lattice
sequence in $V$ with $\fa_0(\L_m)\subset\fa_0(\L)$ minimal amongst
(self-dual) $\oF$-orders normalized by $\Ex$, and $n_m=-\nu_{\L_m}(\b)$;
\item in case \eqref{2plus0}, we take $\L_m^2$ a self-dual
$\oF$-lattice sequence in $V$ with $\fa_0(\L_m^2)\subset\fa_0(\L^2)$
minimal amongst (self-dual) $\oF$-orders, $\L_m=\L^1\perp\L_m^2$, and
$n_m=-\nu_{\L_m}(\b)$.
\end{itemize}
In each case, we have the subgroups $\oH^1_m=H^1(\b,\L_m)$,
$\oJ^1_m=J^1(\b,\L_m)$ of $\oG$, and we put $H_m^1=\oH_m^1\cap G$
etc. Let $\th_m\in\CC(\b,\L_m)$ be the {\it transfer\/} of $\th$, that
is $\th_m=\t_{\L,\L_m,\b}\th$ in the notation of \cite{S4} \S3.6, and
let $\eta_m$ be the unique irreducible representation of $J_m^1$
containing $\th_m$. We form the group 
$$
\tJ^1=(P_1(\L_m)\cap B) J^1.
$$
Then (see \cite{S5}) there is a unique irreducible representation
$\tilde\eta$ of $\tJ^1$ which extends $\eta$ and such that 
$\tilde\eta$ and $\eta_m$ induce equivalent irreducible representations
of $P_1(\L_m)$. Moreover, if $I_g(\tilde\eta)$ denotes the
$g$-intertwining space of $\tilde\eta$, we have
$$
\dim I_g(\tilde\eta)=\begin{cases} 1 &\hbox{if }g\in \tJ^1 (B\cap
G)\tJ^1; \\  0&\hbox{otherwise.} \end{cases}
$$
A $\beta$-extension of $\eta$ is an irreducible representation $\k$ of
$J$ such that $\k|_{\tJ^1}=\tilde\eta$.

\subsection{The main theorem   }

\begin{Theorem}\label{nongenericscs}  
The non-generic supercuspidal representations of $ \ \Sp_4(F)$ are the following. 
\begin{enumerate} \item The positive level supercuspidal representations  
attached to a skew semisimple stratum \break 
$[\L,n,0,\b]$ as above and such  that:  
\begin{itemize}
\item either $[\L,n,0,\b]$ is a sum of non-zero simple strata (cases (\ref{case4}), (\ref{2then0}) and (\ref{2plus2})) 
and there is no maximal unipotent subgroup of $G$ on which $\psi_\b$ is a character;  

\item or  $[\L,n,0,\b]$ is the sum of a non-zero simple stratum and a null stratum in dimension 2 (case (\ref{2plus0})). 
\end{itemize}

\item \label{level0} The  level zero supercuspidal representations  
induced from the inflation 
to a maximal parahoric subgroup $\mathcal P$ 
of a cuspidal representation $\sigma$ of $ \ \mathcal P/ \mathcal P^1$, 
where $\mathcal P^1$ is the pro-$p$-radical of 
$\mathcal P $, satisfying one of the following: 
\begin{enumerate} 

\item \label{level0nonconnected} $\mathcal P $ is attached to a non-connected 
subset of the extended Dynkin diagram of $G$, that is, 
$\mathcal P/ \mathcal P^1$ is isomorphic to $\Sp(2,k_F) \times \Sp(2,k_F)$. 

\item \label{level0connected} $\mathcal P $ is isomorphic to $ \ \Sp_4(\oF)$  
and $\sigma $ is  a non-regular 
cuspidal representation of $ \ \Sp_4(k_F)$. 

\end{enumerate}
\end{enumerate}
\end{Theorem} 

 For positive level representations the theorem will follow 
 from Proposition \ref{Uderprop}, which establishes the conditions on $\beta$ for there 
 to exist a maximal unipotent subgroup of $G$ on which $\psi_\b$ is a character, and from 
 Theorem \ref{genericscs}
 and section \ref{S6}. 
The proof for level zero representations is given below, 
with a more detailed list.

\subsection{The generic level zero representations of ${\bf Sp}_{\bs 2 \bs N}\bs(\bs F\bs)$ }

Note that for finite reductive groups, the notion equivalent to genericity is called regularity: 
 a representation of $   \Sp_{2N}(k_F)$ is called {\it regular} if it 
contains a nondegenerate character of a maximal unipotent subgroup. 
Part (\ref{level0}) of the above theorem, i.e. the level zero case, 
actually holds for $   \Sp_{2N}(F)$, as a consequence of Propositions  
\ref{BHprop} and \ref{criterion}.

\begin{Proposition} 
  Let $\mathcal P$ be  a maximal parahoric subgroup of $ \ \Sp_{2N}(F)$
   with pro-$p$-radical $\mathcal P^1$ 
and let $\sigma$  be a cuspidal representation of $ \ \mathcal P/ \mathcal P^1$. 
The representation $\pi$ of $ \ \Sp_{2N}(F)$  induced from the inflation of 
$\sigma$ to $\mathcal P$ is irreducible supercuspidal. 
It is generic if and only if the quotient $ \ \mathcal P/ \mathcal P^1$ 
is isomorphic to $ \ \Sp_{2N}(k_F)$ and $\sigma$ identifies to a regular 
cuspidal representation of $ \ \Sp_{2N}(k_F)$. 
\end{Proposition} 

\begin{proof} 
Up to conjugacy, we may assume that $\mathcal P$ is standard; in particular, 
using the notation in \ref{S1criterion},  
$\mathcal P$ contains $I$. Then our standard $\mathcal P$  is the group of invertible and 
symplectic elements in the order 
$$
\mathfrak A = 
 \left(\begin{matrix} M_i(\oF) &  M_{i, \, 2N-i}(\oF) & M_i(\pF^{-1}) \cr 
  M_{2N- \; i, \, i}(\pF) & M_{2N-2i}(\oF) &  M_{2N-\;i, \, i}(\oF) \cr
  M_i(\pF) &  M_{i, \,  2N-i}(\pF) & M_i(\oF) 
 \end{matrix}\right)  \qquad   \text{ for some integer } i, 0 \le i \le N . 
$$
Assume first that $ \ \mathcal P/ \mathcal P^1$ is isomorphic to $ \Sp_{2N}(k_F)$, 
i.e. $i=0$ or $N$, 
 and use Proposition \ref{BHprop}. 
If $\sigma$ is regular,  then $\pi$ is generic. 
Conversely  if $\pi$ is generic, there is a maximal unipotent subgroup  $U'$ 
and a character $\chi'$ of $U'$ such that  
$\sigma$ contains $\chi_{|\mathcal P \cap U'}$. The subgroup $U'$ is conjugate to
the subgroup $U$  of all
upper triangular unipotent matrices  so,  using the Iwasawa decomposition 
$G = \mathcal P \mathcal B$ as in \ref{S1criterion},  we may replace $U'$ by $U$. 
Since $\sigma$ is cuspidal, Lemma \ref{nondegenerate}, applied to 
$   \Sp_{2N}(k_F)$,   tells us that 
$\chi_{|\mathcal P \cap U}$ identifies with a nondegenerate character of 
$ \ \mathcal P \cap U / \mathcal P^1 \cap U$, a maximal unipotent subgroup of 
$ \Sp_{2N}(k_F)$, hence $\sigma$ is regular. 

\medskip 
Assume now that $1 \le i \le N-1$: then  $ \ \mathcal P/ \mathcal P^1$  
is  isomorphic to 
$ \Sp_{2i}(k_F)\times \Sp_{2N-2i}(k_F)$, the relevant entries being those in 
$\left(\begin{matrix} \ast & 0 & \ast \cr 0 &  \ast & 0 \cr \ast &0&\ast
 \end{matrix}\right)$  in the above description of $\mathfrak A$. 
Assume for a contradiction that $\pi$ is generic: from Proposition \ref{criterion}, 
plus the inclusion $I_1\subset \mathcal P$, 
 there exist $w\in W$ 
and a   character $\chi$ of $U$   such that\/ $ \sigma$
contains the character $\chi^w$ of\/ $ \mathcal P \cap U^w$.
 Let $\ov{U^w} = \mathcal P \cap U^w /  \mathcal P^1 \cap U^w$  and  let $\ov{\chi^w}$ be  the character of 
 $\ov{U^w}$  defined by $\chi^w$. The group  $\ov{U^w}$ is a maximal unipotent subgroup of 
  $ \ \mathcal P/ \mathcal P^1  $ (a simple 
  combinatoric argument suffices here). 
We will show that, since  $\chi$ is trivial on $U_{\text der}$, the character 
$\ov{\chi^w}$ is degenerate, thus contradicting the cuspidality of $\sigma$.  

To fix ideas, suppose that $w=1$. 
 The intersection of the image of $U$ with $\Sp_{2i}(k_F)$ is the subgroup 
 $\bar U_i$ of  
upper triangular unipotent matrices while the image of $U_{\text der}$ 
contains the simple long root: the restriction of $\bar\chi$ 
to $\bar U_i$ is degenerate hence 
$\sigma$ cannot be cuspidal (Lemma \ref{nondegenerate}).

In general, observe that $w$ must map the $N$ positive long roots 
(corresponding to 
the antidiagonal entries in $U$) onto a set $\mathcal E$ of 
$N$ long roots that correspond to 
the long root entries in $U^w$. Those $N$ long roots separate into $i$ long roots in 
$\bar U^w \cap \Sp_{2i}(k_F)$ and $N-i$ long roots in $\bar U^w \cap \Sp_{2N-2i}(k_F)$. 
The $N-1$ positive not simple long roots corresponding to antidiagonal entries in 
$U_{\text der}$ are sent onto a subset of $N-1$ long roots in $\mathcal E$:
 only one is missing, 
so either in $\bar U^w \cap \Sp_{2i}(k_F) $ 
or in $\bar U^w \cap \Sp_{2N-2i}(k_F) $, 
the unique long root entry that does not belong to the derived group 
does belong to  $U_{\text der}^w$: on this group, the restriction of $\bar\chi^w$ 
is degenerate, so $\sigma$ is not cuspidal. 
\end{proof}

\subsection{The cuspidal representations of ${\bf Sp}_{\bs 4}\bs(\bs {\mathbb F_q}\bs)$ }\label{cuspfinite}

We  come back to  $\Sp(4)$. It has been known for a long time that 
among the cuspidal representations of $\Sp_4(k_F)$,  
  only one  is non-regular, the famous representation $\theta_{10}$ of 
  Srinivasan (\cite{Sp} II.8.3, \cite{Sr}); 
it is the unique cuspidal unipotent representation of $\Sp_4(k_F)$. 
Yet this common knowledge lacks of a reference in the modern setting 
of Deligne-Lusztig characters, we  thus pause here to detail the list 
of  cuspidal representations of $\Sp(4,\mathbb F_q)$ as they arise from the 
Lusztig classification. The necessary background and notations are taken from the book 
\cite{DM}, specially chapter 14. 

\medskip 

For this subsection only we let $G =  \Sp(4,\bar{\mathbb F}_q)$ 
and we let $F$ be the standard Frobenius on $G$, acting as $x \mapsto x^q$ 
on each entry, so that $G^F = \Sp(4,  \mathbb F_q)$. We let $G^\ast$ be the dual group 
$SO(5, \bar{\mathbb F}_q)$ with  standard Frobenius $F^\ast $.

Deligne-Lusztig characters of $G^F$ are parameterized by pairs $(T^\ast, s)$, 
$T^\ast$ an $F^\ast$-stable maximal torus of $G^\ast$ and $s$ an element of 
${T^\ast}^{F^\ast}$, up to 
${G^\ast}^{F^\ast}$-conjugacy. 
A rational series of irreducible characters of $G^F$ is made of all irreducible 
components of Deligne-Lusztig characters $R_{{T}^\ast}^G (s)$ where 
the rational  conjugacy class of $s$ (i.e. the ${G^\ast}^{F^\ast}$-conjugacy class of $s$) is fixed. Rational series
of characters are disjoint 
and exhaust irreducible characters of $G^F$. 
Cuspidal (irreducible) characters are those characters that 
appear in some $R_{{T}^\ast}^G (s)$ for a minisotropic torus $T^\ast$ 
and do not appear in any $R_{{T}^\ast}^G (s)$ where the torus $T^\ast$ 
is contained in a proper $F^\ast$-stable Levi subgroup of $G^\ast$. 

\medskip

Let $s$ be  a rational semi-simple element contained
 in an $F^\ast$-stable maximal torus $T^\ast$ of $G^\ast$, let $C_{G^\ast}(s)$ 
 (resp. $C_{G^\ast}^o(s) $) be its centralizer in $G^\ast$ (resp. the connected component 
 of its centralizer) and let $W(s)$ (resp. $W^o(s)$) be the Weyl group of $C_{G^\ast}(s)$ 
 (resp. $C_{G^\ast}^o(s) $) relative to $T^\ast$, contained in the Weyl group  $W(T^\ast)$
 of $G^\ast$ relative to $T^\ast$.

For $w$ in $W(T^\ast)$, there exists  an $F^\ast$-stable maximal torus $T_w^\ast$ of $G^\ast$
of type $w$ with respect to $T^\ast$ and containing  $s$ 
if and only if $w$ belongs to $W^o(s)$. 
Letting $x$ be the type of $T^\ast$ with respect to some split torus, 
one defines by the formula 
$$
\chi(s) = (-1)^{l(x)} |W^o(s)|^{-1}  \sum_{w \in W^o(s)} (-1)^{l(w) } R_{T_w^\ast}^G (s)   
$$
a proper character 
$\chi(s)$ which is a multiplicity one sum of regular irreducible characters, each appear\-ing 
with multiplicity  $\pm 1$   in Deligne-Lusztig characters  underlying the series attached to $s$. 
We have  $$ \aligned  
  &\left\langle R_{T_w^\ast}^G (s),  R_{T_w^\ast}^G (s) \right\rangle_{G^F} =
   \text{Card }   W(s)  ^{w  F^\ast}  \  \text{ and }   \ 
 \left\langle \chi(s), \chi(s)\right\rangle_{G^F} = |(W(s)/W^o(s))^{F^\ast}|   \endaligned $$ 
and the results in {\it loc. cit.},  \S 14,   
  imply the following, for the 
rational series of characters  attached to  $s$: 
\begin{enumerate}
	\item\label{mini} 
	If only minisotropic rational maximal tori contain $s$, all characters in the series are cuspidal. 
	The number of regular cuspidal characters in the series is the number of components of $\chi(s)$.  
	\item\label{nomini} 
	If no minisotropic rational maximal torus contains $s$, there is no cuspidal  character  in the series. 
	\item\label{mixed}
	  If at least one minisotropic rational maximal torus and at least one non-minisotropic rational maximal torus contain $s$,
	no cuspidal character in the series (if any) is regular.
	\item\label{onemini} 
	If exactly one minisotropic rational maximal torus 
	(up to rational conjugacy) and at least one non-minisotropic rational maximal torus contain $s$,
	 there is no cuspidal  character  in the series.  
\end{enumerate}

The Weyl group of $G^\ast$ has $8$ elements. A  rational maximal torus of type $w$ 
	 with respect to a split torus is minisotropic if and only if $w$ is either a Coxeter element $h$
	 (there are two of them, conjugate in the Weyl group) or the element of maximal length $w_0$. 
	 Rational points of such a torus are conjugate to ${T_0^\ast}^{w F^\ast}$, isomorphic 
	  to $$\aligned 
	  &\mathbb K_4^2 = \ker N_{{\mathbb F}_{q^4}^\times/ {\mathbb F}_{q^2}^\times} \qquad \qquad \text{  for } w=h , \\
	 &\mathbb K_2^1 \times \mathbb K_2^1 = 
	 \ker N_{{\mathbb F}_{q^2}^\times/ {\mathbb F}_q^\times} \times \ker N_{{\mathbb F}_{q^2}^\times/ {\mathbb F}_q^\times} 
	\quad  \text{ for } w=w_0. \endaligned$$ 
	 
	 \medskip 
	 
   The table below lists the families of 
   geometric conjugacy classes of rational semi-simple elements of $G^\ast$ 
 through a representative $s_0$ (not necessarily rational) in   the diagonal torus 
 
$$ T_0^\ast    = \{ \  t^\ast(\lambda, \mu) =  \left(\begin{smallmatrix} \lambda & & & &  \cr & \mu & & & \cr 
 &&1&&
 \cr &   & & \mu^{-1} & \cr  & & &&  \lambda^{-1} \end{smallmatrix}\right)
\ / \  \lambda,   \mu \in  \bar {\mathbb F}_q^\times 
   \} , 
   $$ 
   
using the following notation:  we fix 
	 $\zeta_4$, a primitive $(q^4-1)$-th root of unity in  $\bar {\mathbb F}_q^\times$,  
	and we let  
	$$\zeta_{4,2} = \zeta_4^{q^2-1}, \quad \zeta_2 = \zeta_4^{q^2+1}, \quad  
	  \zeta_{2,1} = \zeta_2^{q-1} \quad  \text{ and  } \   \zeta = \zeta_2^{q+1}.$$ 
	  
	  \bigskip
	 
	 	In cases 13 and 14,   $\chi(s) = R_{T^\ast}^G (s)$  is irreducible, cuspidal and regular (it actually 
contains any nondegenerate character of a maximal unipotent subgroup) (\ref{mini}); we get 
   $\ \dfrac {(q-1)(q-3)} 8 + \dfrac {(q^2-1)}  4 \ $    equivalence  classes of such representations. 
   
   \bigskip 
   
	Cases  4, 5, 6, 7, 8 and 9 give no cuspidal representations   (\ref{nomini}), neither do 
 cases 10 and 12 (\ref{onemini}). Cases  2 and 3 each determine two rational series, in which again 
 (\ref{onemini}) applies: they contain no cuspidal. 
 
 \bigskip
 
 Missing   cuspidals  (\cite{Sp} II.8.3, \cite{Sr}) now must come from cases 1 and 11. Indeed 
 case 11 produces two rational series, one of which satisfying (\ref{nomini}),  but the other (\ref{mini}), 
 for a torus of type $w_0$. Here $\chi(s) = R_{T^\ast}^G (s)$  is the sum of two irreducible,  cuspidal and regular (but 
 for different choices
	of  a nondegenerate character of a maximal unipotent subgroup) representations and we get 
 $ \  2 \  \dfrac {(q-1)} 2 \ $       equivalence  classes of such representations. 
 
 \bigskip 
 
	 Last, case 1 gives the so-called unipotent series, which for $ \Sp(4,  \mathbb F_q)$ 
	 contains exactly one cuspidal representation (\cite{L}, Theorem 3.22), non-regular by (\ref{mixed}).

\newpage

	$$
	\vbox{\offinterlineskip\halign{&\vrule#&   \strut\quad\hfil#\quad
	\cr
		\noalign{\hrule}height4pt&\omit&&\omit&&\omit&&\omit&&\omit&&\omit& 
	\cr
		 height4pt&\omit&&\omit&&\omit&&\omit&&\omit&&\omit& 
	\cr 
		&\hskip-2.5mmCase\hskip-2.5mm&&  $s_0$  &&condition&&number&&$W^o(s_0) $ &&\hskip-2.5mm$|W(s_0)/W^o(s_0)|$\hskip-2.5mm & 
		\cr
		 height4pt&\omit&&\omit&&\omit&&\omit&&\omit&&\omit& 
	\cr
	\noalign{\hrule}height4pt&\omit&&\omit&&\omit&&\omit&&\omit&&\omit& 
	 \cr
	  height4pt&\omit&&\omit&&\omit&&\omit&&\omit&&\omit& 
	  \cr
	  &1&&$ t^\ast(1, 1) $&& &&$1$&&$W$&&$1$& 
	  \cr
	   height4pt&\omit&&\omit&&\omit&&\omit&&\omit&&\omit&  
	   \cr
	   \noalign{\hrule}height4pt&\omit&&\omit&&\omit&&\omit&&\omit&&\omit& 
	 \cr
	  height4pt&\omit&&\omit&&\omit&&\omit&&\omit&&\omit& 
	  \cr
	  &2&&$ t^\ast(-1, -1) $&& &&$1$&&$\left\langle s_{\alpha'}, s_{\alpha'+2\beta'}\right\rangle$&&$ 2$& 
	  \cr
	   height4pt&\omit&&\omit&&\omit&&\omit&&\omit&&\omit&  
	   \cr
	   \noalign{\hrule}height4pt&\omit&&\omit&&\omit&&\omit&&\omit&&\omit& 
	 \cr
	  height4pt&\omit&&\omit&&\omit&&\omit&&\omit&&\omit& 
	  \cr
	  &3&&$ t^\ast(-1, 1) $&& &&$1$&&$\left\langle  s_{\beta'}\right\rangle$&&$ 2$& 
	  \cr
	   height4pt&\omit&&\omit&&\omit&&\omit&&\omit&&\omit&  
	   \cr
	   \noalign{\hrule}height4pt&\omit&&\omit&&\omit&&\omit&&\omit&&\omit& 
	 \cr
	  height4pt&\omit&&\omit&&\omit&&\omit&&\omit&&\omit& 
	  \cr
	  &4&&$ t^\ast(\zeta^i , 1  ) $&&$\zeta^i \ne \pm 1  $&&$\dfrac {q-3} 2 $&&$\left\langle  s_{\beta'}\right\rangle$&&$1$& 
	  \cr
	   height4pt&\omit&&\omit&&\omit&&\omit&&\omit&&\omit&  
	   \cr
	    \noalign{\hrule}height4pt&\omit&&\omit&&\omit&&\omit&&\omit&&\omit& 
	 \cr
	  height4pt&\omit&&\omit&&\omit&&\omit&&\omit&&\omit& 
	  \cr
	  &5&&$ t^\ast(\zeta^i, -1 ) $&&$\zeta^i \ne \pm 1  $&&$\dfrac {q-3} 2 $&&$1$&&$ 2$& 
	  \cr
	   height4pt&\omit&&\omit&&\omit&&\omit&&\omit&&\omit&  
	   \cr
	    \noalign{\hrule}height4pt&\omit&&\omit&&\omit&&\omit&&\omit&&\omit& 
	 \cr
	  height4pt&\omit&&\omit&&\omit&&\omit&&\omit&&\omit& 
	  \cr
	  &6&&$ t^\ast(\zeta^i, \zeta^i) $&&$\zeta^i \ne \pm 1  $&&$\dfrac {q-3} 2 $&&$\left\langle s_{\alpha'}   \right\rangle$&&$1$& 
	  	\cr
	   height4pt&\omit&&\omit&&\omit&&\omit&&\omit&&\omit&  
	   \cr
	    \noalign{\hrule}height4pt&\omit&&\omit&&\omit&&\omit&&\omit&&\omit& 
	\cr 
	height4pt&\omit&&\omit&&\omit&&\omit&&\omit&&\omit& 
	\cr
&\omit&&\omit&&$\zeta^i \ne \pm 1   $&&\omit&&\omit&&\omit& 
	\cr
	&7&&$ t^\ast(\zeta^i, \zeta^j) $&&$  \zeta^j \ne \pm 1 $&&$\dfrac {(q-3)(q-5)} 8 $&&$ 1$&& $1$& 
\cr
&\omit&&\omit&&$\zeta^i \ne    \zeta^{\pm j}   $&&\omit&&\omit&&\omit& 
	 \cr
	 \noalign{\hrule}height4pt&\omit&&\omit&&\omit&&\omit&&\omit&&\omit& 
	 \cr
	  height4pt&\omit&&\omit&&\omit&&\omit&&\omit&&\omit& 
	   \cr
	     &\omit&&\omit&&$\zeta_2^i \notin \mathbb K_2^1 $ && && && & 
	  \cr
	  &8&&$ t^\ast(\zeta_2^i, \zeta_2^{qi}) $&&$\zeta_2^i \notin \mathbb F_q^\times $ && $\dfrac {(q-1)^2} 4$&&$1$ && $1$& 
	   \cr
	   height4pt&\omit&&\omit&&\omit&&\omit&&\omit&&\omit&  
	   \cr
	   \noalign{\hrule}height4pt&\omit&&\omit&&\omit&&\omit&&\omit&&\omit& 
	   \cr
	    height4pt&\omit&&\omit&&\omit&&\omit&&\omit&&\omit& 
	   \cr
	     &\omit&&\omit&&$\zeta^i \ne \pm 1 $ && && && & 
	    \cr
	      &9&&$ t^\ast(\zeta^i,\zeta_{2,1}^j ) $&&  $\zeta_{2,1}^j \ne \pm 1 $&& $\dfrac {(q-1)(q-3)} 4$ &&$1$&& $1$& 
	    \cr
	     height4pt&\omit&&\omit&&\omit&&\omit&&\omit&&\omit& 
	     \cr
	        \noalign{\hrule}height4pt&\omit&&\omit&&\omit&&\omit&&\omit&&\omit& 
	 \cr
	  height4pt&\omit&&\omit&&\omit&&\omit&&\omit&&\omit& 
	  \cr
	  &10&&$ t^\ast(1, \zeta_{2,1}^{i}) $&&$\zeta_{2,1}^i \ne \pm 1$&& $\dfrac {(q-1)} 2 $&&$\left\langle  s_{\alpha'+\beta'}\right\rangle$ &&$1$ & 
	  \cr
	   height4pt&\omit&&\omit&&\omit&&\omit&&\omit&&\omit&  
	   \cr
	        \noalign{\hrule}height4pt&\omit&&\omit&&\omit&&\omit&&\omit&&\omit& 
	 \cr
	  height4pt&\omit&&\omit&&\omit&&\omit&&\omit&&\omit& 
	  \cr
	  &11&&$ t^\ast(-1, \zeta_{2,1}^{i}) $&&$\zeta_{2,1}^i \ne \pm 1$&&$ \dfrac {(q-1)} 2 $&&$1$ &&$ 2$ & 
	  \cr
	   height4pt&\omit&&\omit&&\omit&&\omit&&\omit&&\omit&  
	   \cr
	        \noalign{\hrule}height4pt&\omit&&\omit&&\omit&&\omit&&\omit&&\omit& 
	 \cr
	  height4pt&\omit&&\omit&&\omit&&\omit&&\omit&&\omit& 
	  \cr
	  &12&&$ t^\ast(\zeta_{2,1}^i, \zeta_{2,1}^{-i}) $&&$\zeta_{2,1}^i \ne \pm 1$&&$ \dfrac {(q-1)} 2 $&&$\left\langle  s_{\alpha'+2\beta'}\right\rangle$ &&$1$ & 
	  \cr
	   height4pt&\omit&&\omit&&\omit&&\omit&&\omit&&\omit&  
	   \cr
	       \noalign{\hrule}height4pt&\omit&&\omit&&\omit&&\omit&&\omit&&\omit& 
	       \cr
	        height4pt&\omit&&\omit&&\omit&&\omit&&\omit&&\omit& 
	           \cr
&\omit&&\omit&&$\zeta_{2,1}^i \ne \pm 1   $&&\omit&&\omit&&\omit& 
	        \cr
	        &13&&$ t^\ast(\zeta_{2,1}^i, \zeta_{2,1}^j) $&&$\zeta_{2,1}^j \ne \pm 1   $&&$\dfrac {(q-1)(q-3)} 8 $  &&$1$ &&$1$ & 
	        \cr
&\omit&&\omit&&$\zeta_{2,1}^i \ne \zeta_{2,1}^{\pm j}   $&&\omit&&\omit&&\omit& 
	        \cr
	           height4pt&\omit&&\omit&&\omit&&\omit&&\omit&&\omit& 
	         \cr
	          \noalign{\hrule}height4pt&\omit&&\omit&&\omit&&\omit&&\omit&&\omit& 
	     \cr
	      height4pt&\omit&&\omit&&\omit&&\omit&&\omit&&\omit& 
	      \cr
	   	   &14&&$ t^\ast(\zeta_{4,2}^{qi}, \zeta_{4,2}^{i}) $&&$\zeta_{4,2}^{i} \ne \pm 1$&&   
  $\dfrac {q^2-1} 4 $ && $1$ && $1$ & 
	      \cr
	       height4pt&\omit&&\omit&&\omit&&\omit&&\omit&&\omit& 
	       \cr
	         \noalign{\hrule}\noalign{\bigskip} 
	         \cr}}$$

\setcounter{section}{2}

%%%%%%%%%%%%%%%%%%%%%%%%%%%%%%%%%%%%%%%%%%%%%%%%%%%%%%%
\Section{The function ${\bs \psi}_{\bs \b}$ on maximal unipotent subgroups}\label{S4}
%%%%%%%%%%%%%%%%%%%%%%%%%%%%%%%%%%%%%%%%%%%%%%%%%%%%%%%

 A key step in the determination of Whittaker functions in $\GL_N(F)$ in  \cite{BH}  
 is the construction of a maximal unipotent subgroup $U$ of  $\GL_N(F)$ 
 on which $\psi_\b$ defines a character (\it loc. cit.  \rm propositions 2.1 and 2.2). 
This will be a key step indeed  in the determination of generic supercuspidal representations of $\Sp_4(F)$: 
the existence of such a subgroup on which $\psi_\b$ defines a {\bf non degenerate} character will turn out to be a sufficient condition 
for genericity (see \S \ref{S5}), whereas the non existence will imply non genericity (see \S \ref{S6}). In cases where such a $U$ exists but 
$\psi_\b$ is degenerate, we will find both generic and non generic representations.

\subsection{The quadratic form ${\bs h}{\bs (}{\bs v}{\bs ,} {\bs \b }{\bs  v}{\bs )}$ \label{quadratic} }
 
 \begin{Proposition}\label{flag} 
 Let $\b$ be an element of $A$ such that $\bar \b = - \b  $ and let $\psi_\b$ be the function on $G$ defined by 
 $\psi_\b (x) = \psi\left( \tr (\b (x-1))\right)$, $x \in G$. 
 The following are equivalent: 
 
 \begin{enumerate}
 \item There   exists a   maximal unipotent subgroup $  U$ of $  G$ such that the restriction of $\psi_\b$ to $U$ 
 is a character of $U$. 
\item There   exists a   maximal unipotent subgroup $  U$ of $  G$ such that $\psi_\b (x)=1$ for all $x \in    U_{\text{\rm der}}$. 
\item  There exists a totally isotropic  flag of subspaces of $V$: 
$$\{0\} \subset V_1 \subset V_2 \subset V_3 \subset V$$ 
such that $\b V_i \subset V_{i+1}$ for $i=1,2$.  
\item  The quadratic form $v \mapsto h(v, \b  v) $ on $V$ has non trivial isotropic vectors.   
\end{enumerate}
 
\end{Proposition}
 
 \begin{proof}
 The equivalence of the first three statements is straightforward and a variant of 
 \cite{BH}   2.1; note that  a maximal flag of subspaces of $V$ determines a maximal unipotent subgroup of $G$ if and only if it is totally isotropic.
 
  Certainly (iii) implies (iv):  a basis vector $v$ for $V_1$ satisfies $h(v, \b  v)=0$ since $V_2$ is its own orthogonal. 
 Assuming (iv), let $v$ be 
    a non-zero vector   in $V$ such that   $h(v, \beta v) =0$ and  put $V_1= {\text{\rm Span }} \{v\}$, 
    $V_3 = V_1^\perp$. If $\b v$ and $v$ are colinear, let $V_2$ be any totally isotropic $2$-dimensional subspace of $V$ containing $V_1$, otherwise put 
     $V_2= {\text{\rm Span }} \{v, \beta v\}$: (iii) is satisfied since, for a totally isotropic flag as in (iii), the conditions 
    $\b V_1 \subset V_{2}$ and  $\b V_2 \subset V_{3}$ are equivalent (recall  
   $\bar\b  = -\b  $).   
 \end{proof}
 
  \begin{Remark}\label{flagdegenerate} Assume the conditions in Proposition \ref{flag}  hold and let 
  $U$ be a maximal unipotent subgroup  of $  G$ attached to a totally isotropic flag  $\{0\} \subset V_1 \subset V_2 \subset V_3 \subset V$  
such that $\b V_i \subset V_{i+1}$ for $i=1,2$.  A simple inspection shows that the character $\psi_\b$ of  $U$ is non degenerate if and only if 
$\b V_1 $ is not contained in $  V_1$ and $\b V_2 $ is not contained in $  V_2$. 
  \end{Remark}
 
 We need to investigate those cases where the element $\b$ appears in a skew semi-simple stratum $[\Lambda, n, 0 , \beta]$  as 
 listed  in \S\ref{construction}. 
We need an extra piece of notation in cases \ref{case4} or \ref{2then0},  where the stratum is simple: the field extension $E=F[\b]$ has degree $4$ or $2$; 
  we let $E_0$ be the field of fixed points of the involution $x \mapsto \bar x$ on $E$, so that 
  $[E:E_0]=2$,  and we define a skew-hermitian form  $\d$ on $V$ relative to $E/E_0$   by 
  \begin{equation}
h(av,w)= \trEF(a \d(v,w)) \  \text{ for all } a \in E, v,w \in V . \label{skewform}
\end{equation}
(This notation will also be used in case \ref{2plus2} when $E_1$ and $E_2$ are isomorphic, with $E=E_1$.)
The determinant of $\d$ belongs to $F^\times$ if $[E:F]=2$; it is a skew element in $E^\times$ if $[E:F]=4$.

\begin{Proposition}\label{Uderprop} 
Let $\b$ be an element of $A$ appearing in a skew semi-simple stratum 
 $[\L, n, 0 , \beta]$   as in  \S\ref{construction}. The only cases in which there does not exist 
   a   maximal unipotent subgroup $U$ of $G$ on which  $\psi_\b $ is a character are the following: 

\begin{enumerate}
\item   The element $\b$ generates a biquadratic extension  $E=F[\b]$ of  $F$ (case \ref{case4})   and the coset  
$  \beta \,  \det (\d )  \,  N_{E/E_0}(E^\times)$ in $E_0^\times$   is the  $ N_{E/E_0}(E^\times)$-coset  
 that does not contain the kernel of $ \, {\hbox{\rm tr}}_{E_0/F}  $. 

\item  The element $\b$ generates a  quadratic extension  $E=F[\b]$ of  $F$    (case \ref{2then0})    and  the skew-hermitian form 
$\d$ on $V$, $\dim_E V=2$, is anisotropic -- that is,  $ \  \det (\d)  \notin N_{E/F}(E^\times)$. 

\item   The symplectic space $V$ decomposes as $V = V^1  \bot V^2$  and the element $\b $ decomposes accordingly as 
$\b = \b_1 + \b_2$ where for $i=1,2$, $\b_i$ generates a quadratic extension   $E_i=F[\b_i]$ (case \ref{2plus2}),   $E_1 $ is isomorphic to $E_2 $   and    
$ \ \b_1/\b_2 \notin  \,  \det (\d) \,  N_{E/F}(E^\times).$  
\end{enumerate}
\end{Proposition}

 \begin{Remark}\label{Uderpropdegenerate} Let $\b$ be  as above. Assume there  exists 
   a   maximal unipotent subgroup $U$ of $G$ on which  $\psi_\b $ is a character.  Then:
\begin{itemize}
	\item in cases \ref{case4} and  \ref{2plus2} the character $\psi_\b $ of $U$  is   non degenerate;  
	\item in cases  \ref{2then0} and \ref{2plus0} the character $\psi_\b $ of $U$  is    degenerate. 
\end{itemize}
  \end{Remark}

 The proof of those statements occupies the next two subsections.  We recall that,  
  up to isomorphism, there is exactly one   anisotropic   quadratic form on $V$:  its determinant is a square and its Hasse-Minkowski symbol is equal to  $-(-1,-1)_F$ (\cite{O}, \S 63C).

\subsection{The biquadratic extension}\label{biquad}

Let us  examine the  case  of a maximal simple stratum (case \ref{case4}).  The determinant of the  quadratic form $v \mapsto h(v, \b  v) $ on $V$ is the determinant of $\beta$, i.e. $N_{E/F}(\beta)$. 

\begin{Lemma}   The norm $N_{E/F}(\beta)$ of $\beta$ is a square in $F^\times$ if and only if   $E$ is biquadratic.   If this holds we have: $N_{E/E_0}(E^\times) =  {F^\times} {E_0^\times}^2$. 
\end{Lemma}

\begin{proof}  A four-dimensional extension of $F$ is called \it biquadratic \rm if it is Galois with Galois group $\mathbb Z/2 \mathbb Z \times \mathbb Z/2 \mathbb Z$. Biquadratic extensions of $F$ are all isomorphic, their  norm subgroup is $F^{\times 2}$(class field theory). The ``if'' part is 
then clear.  Now assume that $N_{E/F}(\beta)$ is a square.  Since $\beta$ is skew and generates a degree four field extension of $F$,    its square $\beta ^2$ generates $E_0$ over $F$ and is not a square in $E_0^\times$, while $ N_{E_0/F}(\beta^2)=N_{E/F}(\beta)$ must be a square in $F^\times$.
We proceed according to ramification.

  If $E_0$ is ramified over $F$: $\beta^2$ must have even valuation, its squareroot generates an unramified extension of $E_0$. So $E/E_0$ is unramified and $N_{E/E_0}(E^\times) $ is made of even valuation elements, i.e. is equal to ${F^\times} {E_0^\times}^2$ since $\mathfrak o_{E_0}^\times = \mathfrak o_{F}^\times (1+\mathfrak p_{E_0})$. It follows that $N_{E/F}(E^\times) =  {F^\times}^2$.

   If   $E_0$ is unramified over $F$, we have in the residual field $k_{E_0}$: 
 
 \centerline{ \it $u \in k_{E_0}^\times$ is a square in $k_{E_0}^\times$ if and only if $N_{k_{E_0}/k_F}(u)$ is a square in $k_{F}^\times$. }

 We write $\beta^2 = \varpi_F^{j}u$ with $u \in \mathfrak o_{E_0}^\times$. Then $ N_{E_0/F}(\beta^2)= \varpi_F^{2j} N_{E_0/F}(u)$. It follows that $u $   must be a square and   $\beta^2$ must have odd valuation:    its square root generates a ramified extension of $E_0$.   Then $E$ is the extension $E_0[\alpha]$ where $\alpha^2 $ is a uniformizing element in $F$,  and $N_{E/E_0}(E^\times) =  (-\alpha^2)^\mathbb Z  \mathfrak o_{E_0}^{\times 2}=  {F^\times} {E_0^\times}^2$ because $k_{F}^\times \subset k_{E_0}^{\times 2}$. It  follows that  $N_{E/F}(E^\times) =  {F^\times}^2$. 
\end{proof}
  
If $N_{E/F}(\beta)$   is not a square, we are done. Assume from now on that  $E$ is biquadratic   and use the form $\d$ defined in 
\ref{skewform}.  
  Since $\beta$ is skew  and  $\d$ is skew-hermitian, the element $\beta  \d(v,v)$ belongs to $E_0$. Since $V$ is one-dimensional over $E$ the form $\d(v,v)$  is anisotropic and the subset $D(V)= \{\beta \d(v,v) \ / \ v \in V, v \ne 0\} $ of $E_0^\times$  is one of the two cosets of  $N_{E/E_0}(E^\times) $ in $E_0^\times$. On the other hand, we have $N_{E/E_0}(E^\times) =  {F^\times} {E_0^\times}^2$ hence the set of non-zero elements in $\text{Ker } \trEOF $ is fully contained in one of those two cosets. Proposition \ref{Uderprop} follows in this case (and Remark \ref{Uderpropdegenerate} directly follows 
from  \ref{flagdegenerate} since $\b$ generates a degree $4$ extension of $F$). 
  
  \medskip
  \bf Remark. \rm We can be more precise about this condition:
if $E_0$ is unramified over $F$ and $|k_F| \equiv 3 \,  [4]$, then $  h(\b  v,v)$ is anisotropic if and only if $\b  \d(v,v) \notin   {F^\times} {E_0^\times}^2$;   
 otherwise $  h(\b  v,v)$ is anisotropic if and only if $\b   \d(v,v) \in   {F^\times} {E_0^\times}^2$.

\subsection{Cases II, III and IV}\label{other}

The case numbered \ref{2plus0} in \S\ref{construction} is obvious. The set of isotropic vectors for the quadratic 
form $h(v, \b  v) $ is  the subspace $V^2$.    The flags that satisfy  \ref{flag}(iii) are the flags that can be 
written in the form 
 $\{0\} \subset Fe_{-2}\subset Fe_{-2} + Fe_{-1} \subset Fe_{-2} +Fe_{-1}+Fe_{1} \subset V$
 where $\{e_{-i}, e_i\}$ is a 
  symplectic basis   of $V^i$ for $i=1,2$.

 Case \ref{2then0} is also quite clear: as in case \ref{case4} the element $\beta  \d(v,v)$ belongs to $E_0=F^\times$ hence 
 $ h(\b v,v)=   2 \b \d(v,v)$  has isotropic vectors if and only if $\d(v,v)$ does. 
 Furthermore a flag $\{0\} \subset V_1 \subset V_2 \subset V_3 \subset V$ as in \ref{flag} 
 must have the form $V_1 = F v$, where $v$ is non zero and isotropic for $\d$, and $V_2 = <v, \b v>$. Since 
 $\b^2$ belongs to $F^\times$ we always have $\b V_2 = V_2$ so, if $\psi_\b $ defines a character of the corresponding unipotent subgroup of $G$, this 
 character is degenerate. 

 \smallskip 
	We finish with case \ref{2plus2}. 
We have $V = V^1 \bot V^2 $ and 
	 $\beta =\b_1+\b_2$. For $v \in V$,  writing $v=v_1+v_2$ on      $V = V^1 \bot V^2 $, we get $h(v, \b v) = h(v_1, \b_1 v_1) + h(v_2, \b_2 v_2) $. The    determinant of this form is  the product $N_{E_1/F}(\b_1) N_{E_2/F}(\b_2) $. For the form to be anisotropic, the determinant must be a square hence 
    $$N_{E_1/F} (\b_1) \equiv  N_{E_2/F} (\b_2)   \text{ mod }  F^{\times2}.$$
    Each $\b_i$ is skew with characteristic polynomial    $X^2 - (- N_{E_i/F} (\b_i) )$:  the class of $- N_{E_i/F} (\b_i)$ mod the squares  determines, up to isomorphism, the extension $E_i$. So if $E_1$ and $E_2$ are not isomorphic we are done.  
    
    We pursue assuming they are and let $E=E_1 \simeq E_2$. We may see $V$ as a vector space over $E$ and define $\d$ as in 
  \ref{skewform}.  The decomposition    $V = V^1 \bot V^2 $ is orthogonal for $\d$ as well, and 
 for $i=1, 2$, $v_i \mapsto  h(v_i, \b_i v_i)$  is an anisotropic  quadratic form on $V^i$: a (non zero) isotropic vector  for 
 $h(v, \b v)$ must  have the form $v=v^1+v^2$ with $v^i \in V^i$, $v^i \ne 0$. We then have: 
 $$ h(v, \b v)=  h(v^1, \b_1 v^1) + h(v^2, \b_2 v^2) = 2 \,  \b_1 \, \d(v^1,v^1) + 2 \,  \b_2 \,  \d(v^2,v^2),$$
 and \ref{Uderprop} follows. 
     
     The remark on the non degeneracy of $\psi_\b$ on $U$ whenever it defines a character follows from the fact that 
     $V^1$ and $V^2$ are anisotropic for $h(v, \b v)$: the corresponding totally isotropic flag has the form (with notations as above)  
  $\{0\} \subset V_1=<v^1+v^2> \subset V_2 = <v^1+v^2,\b_1 v^1+\b_2 v^2 > \subset V_1^\perp \subset V$. We never have $\b V_2  \subseteq  V_2$ unless 
  $ \b_1^2=\b_2^2$ -- but this belongs to case \ref{2then0}, not to case \ref{2plus2}. 
    
\setcounter{section}{3}

 %%%%%%%%%%%%%%%%%%%%%%%%%%%%%%%%%%%%%%%%%%%%%%%%%%%%%%%
\Section{Generic representations}\label{S5}
%%%%%%%%%%%%%%%%%%%%%%%%%%%%%%%%%%%%%%%%%%%%%%%%%%%%%%%

In section~\ref{S4}, we have found necessary and sufficient conditions
for there to exist a maximal unipotent subgroup $U$ of $G$ on which
$\psi_\beta$ defines a character. When there is such a $U$, this gives
us a candidate for trying to build a Whittaker model, as is the case
in $\GL_N(F)$ (see~\cite{BH}). In this section we consider the
case where there is such a $U$.

\subsection{Characters and $\boldsymbol\b$-extensions}\label{Sgen.1}

\begin{Proposition}[{cf.~\cite[Lemma~2.10]{BH}}] \label{thetaagreeswithpsi}
Let $[\L,n,0,\b]$ be a skew semisimple stratum in
$A$. Let $\th\in\Clob$ be a skew semisimple character and let $U$ be a
maximal unipotent subgroup of $G$ such that $\psi_\b|_{U_{der}}=1$. Then
$$
\th|_{H^1\cap U} = \psi_\b|_{H^1\cap U}.
$$
\end{Proposition}

Unfortunately, we have been unable to find a unified proof of this Proposition; the proof is therefore rather ugly, on a case-by-case basis, and we postpone it to the appendix.

We continue with the notation of Proposition~\ref{thetaagreeswithpsi}. Then we 
can define a character $\Th_\b$ of $\tH^1=(J\cap U)H^1$ by
$$
\Th_\b(uh)=\psi_\b(u)\th(h),\qquad\hbox{for }u\in J\cap U,\ h\in H^1.
$$
Notice that this is a character, since $J$ normalizes $H^1$ and
intertwines $\th$ with itself.

\begin{Corollary}[{cf.~\cite[Lemma~2.11]{BH}}] \label{etacontainspsi}
Let $\eta$ be the unique irreducible
representation of $J^1$ which contains $\th$. Then the restriction of
$\eta$ to $J^1\cap U$ contains the character ${\psi_\b}|_{J^1\cap U}$.
\end{Corollary}

\begin{proof} The proof is essentially identical to that of~\cite[Lemma~2.11]{BH}. We recall that $\bok_\th(x,y)=\th[x,y]$ defines a
nondegenerate alternating form on the finite group
$J^1/H^1$~(\cite[Proposition~3.28]{S4}).
Notice also that the image of $\tH^1\cap J^1$ in $J^1/H^1$ is a totally
isotropic subspace for the form $\bok_\th$, since $\th$ extends
to a character $\Th_\b$ of $\tH^1$. Now we can construct $\eta$ by
first extending $\Th_\b$ to (the inverse image in $J^1$ of) a maximal
totally isotropic subspace of $J^1/H^1$ and then inducing to $J^1$. In
particular, $\eta$ contains $\Th_\b$ and hence ${\psi_\b}|_{J^1\cap U}$.
\end{proof}

Now put $\tJ^1=(J\cap U)J^1=(P(\L)\cap B^\times\cap U)J^1$. Note that, if $J/J^1$ is anisotropic then $J\cap U=J^1\cap U$ so 
$\tJ^1=J^1$.

\begin{Theorem}[{cf.~\cite[Theorem~2.6]{PS}}] \label{kappainduced}
Let $[\L,n,0,\b]$ be a skew semisimple stratum in
$A$ as listed in \S \ref{construction} and let $\kappa$ be a $\b$-extension 
of $\eta$ to $J$ as described there. Assume there is   a
maximal unipotent subgroup $U$ of $G$ such that $\psi_\b|_{U_{der}}=1$ and 
  use the notation above. 
Then:   
$$
\kappa|_{\tJ^1} \simeq \Ind_{\tH^1}^{\tJ^1}\Th_\b.
$$
\end{Theorem}

In particular, we deduce that the restriction of $\kappa$ to $J\cap U$
contains the character $\psi_\b|_{J\cap U}$ ({cf.}~\cite[Lemma~2.12]{BH}). Moreover, except in case \ref{2plus0}, and case \ref{2then0} when $J/J^1$ is isotropic, this means that the restriction of the simple type $\l$ to $J\cap U$ contains $\psi_\b|_{J\cap U}$, since $\l=\k$.

\begin{proof} The proof is in essence the same as that of~\cite[Theorem~2.6]{PS}, which it may be useful to read first: because of the similarities, we do not give all the details here.

We begin by proving
\begin{equation}
\Ind_{(J^1\cap U)H^1}^{J^1}{\Th_\b}\simeq \eta. \label{etainduced}
\end{equation}
We prove this first in cases \ref{case4} and \ref{2plus2}. Here 
$E^1= \{x \in E / x \bar x = 1\}$ is a maximal torus of $G$. Then $J=E^1J^1$ and $\pi=\cInd_J^G\k$.
Since $J\cap U=J^1\cap U$, Corollary~\ref{etacontainspsi} implies that
$\kappa$ contains $\psi_\b|_{J\cap U}$. Hence $\pi$ contains $\psi_\b$ and, since
$\psi_\b$ is then a nondegenerate character, $\kappa$ contains
$\psi_\b|_{J\cap U}$ with multiplicity one (Proposition \ref{BHprop}). Hence $\Th_\b|_{(J^1\cap U)H^1}$ occurs in
$\eta$ with multiplicity precisely one and \eqref{etainduced} follows (see~\cite[Lemma~2.5]{PS}).
Note that this already gives the Theorem cases \ref{case4} and \ref{2plus2}, since $E^1$ is maximal and $\tJ^1=J^1$ in these cases.

Now we consider the other cases \ref{2then0} and \ref{2plus0}. 
Recall that $U$ is given by a flag $\{0\}\subset V_1\subset 
V_2\subset V_3\subset V$ (see Proposition~\ref{flag}), described in \S \ref{other}. 
What we need here is to 
 define a parabolic subgroup $P_0 $ of $G$, with unipotent radical $U_0$ contained in $U$ and with 
 a specific Levi factor $M_0$ conforming to $\L$ in the sense of \cite[\S 10]{BH0}. 
 We achieve this 
 according to the case as follows. 
\begin{enumerate}
\item[(\ref{2then0})] 
From \S \ref{other}, the unipotent subgroup  $U$ is attached to a flag  
$$\{0\} \subset Fw_{-1}\subset Fw_{-1} + F\b w_{-1} \subset Fw_{-1} + F\b w_{-1} +F\b w_{1} \subset V$$ 
where $\{w_{-1},w_1\}$ is a Witt basis    for $V$ over $E$.  
 We then let $M_0$ be the stabilizer of the decomposition $V = E w_{-1} \oplus E w_1$ and 
$P_0$ be the stabilizer of the flag 
 $\{0\}\subset E w_{-1}\subset V$.

\item[(\ref{2plus0})] Here $U$ is attached to a flag 
 $\{0\} \subset Fe_{-2}\subset Fe_{-2} + Fe_{-1} \subset Fe_{-2} + V^1 \subset V$.  
We can pick $e_2$ so that 
  $\{e_{-2}, e_2\}$  is a symplectic basis 
    of $V^2$  adapted to $\L^2$. We then let $M_0$ be the stabilizer of the decomposition $V = Fe_{-2} \oplus V^1 \oplus Fe_{2} $ and 
$P_0$ be the stabilizer of the flag 
 $\{0\}\subset Fe_{-2}  \subset Fe_{-2} \oplus V^1 \subset V$.
\end{enumerate}

Let $\bs k_\th$ be the form defined  in the proof of Corollary \ref{etacontainspsi}. 
In each case, we have the following properties (see {\cite[\S 10]{BH0}} and \cite[Lemma 5.6, Corollary 5.10]{S5}):
\begin{enumerate}
\item $U$, $H^1$ and $J^1$ have Iwahori decompositions with respect to
  $(M_0,P_0)$; 
\item $J^1\cap U_0/H^1\cap U_0$ is a totally isotropic subspace of
  $J^1/H^1$ with respect to the form $\bs k_\th$;
\item $J^1\cap M_0\cap U/H^1\cap M_0\cap U$ is a maximal totally isotropic subspace of $J^1\cap M_0/H^1\cap M_0$;
\item there is an orthogonal sum decomposition
$$
\frac{J^1}{H^1}\ =\ \frac{J^1\cap M_0}{H^1\cap M_0} \perp 
\left( \frac{J^1\cap U_0}{H^1\cap U_0} \times 
\frac{J^1\cap U_0^-}{H^1\cap U_0^-} \right),
$$
where $U_0^-$ is the unipotent subgroup opposite to $U_0$ relative to $M_0$.
\end{enumerate}
Then $(J^1\cap U)H^1$ has an Iwahori decomposition with respect
to $(M_0,P_0)$ and 
$(J^1\cap U)H^1/H^1 \simeq J^1\cap M_0\cap U/H^1\cap M_0\cap U 
\perp J^1\cap
U_0/H^1\cap U_0$ is a maximal totally isotropic subspace of $J^1/
H^1$. In particular, from the construction of Heisenberg extensions, equation~\eqref{etainduced} follows.

For the final stage, as in the construction of $\b$-extensions, there is an $\oE$-lattice sequence $\L_m$ such that 
$\tJ^1=(P^1(\L_m)\cap B)J^1$, and $P^1(\L^m)$ 
still has an Iwahori decomposition with respect to $(M_0,P_0)$. 
Defining $\tilde\eta$ to be $\Ind_{\tH^1}^{\tJ^1}\Th_\b$, we 
see that $\tilde\eta|J^1=\eta$, and one checks that
$$
\Ind_{\tJ^1}^{P^1(\L_m)}\tilde\eta \simeq \Ind_{J_m^1}^{P^1(\L_m)}\eta_m.
$$
Since this uniquely determines $\tilde\eta$, from the definition 
of $\b$-extension we get that $\k|_{\tJ^1}=\tilde\eta$, and the 
result follows.
\end{proof}

\subsection{The positive level generic supercuspidal representations of ${\bf
Sp}_{\bs 4}\bs(\bs F\bs)$}

Here we show that all the positive level supercuspidal representations which are not in the list of Theorem~\ref{nongenericscs} are indeed generic.

\begin{Theorem}\label{genericscs}
Let $\pi=\cInd_J^G\l$ be a positive level irreducible supercuspidal
representation of $G=\Sp_4(F)$ with underlying skew semisimple stratum
$[\L,n,0,\b]$. Suppose that this stratum is not in case \ref{2plus0} and
that there exists a maximal unipotent subgroup 
$U$ of $G$ such that $\psi_\b|_{U_{der}}=1$. Then $\pi$ is generic. 
\end{Theorem}

\begin{proof} Except in case \ref{2then0} when $J/J^1$ is isotropic, this is immediate from Theorem~\ref{kappainduced}, since $\l|_{\tJ^1}=\k|_{\tJ^1}$ in these cases so $\l|_{J\cap U}$ contains the character $\psi_\b|_{J\cap U}$; hence, by Proposition~\ref{BHprop}, $\pi$ is generic.

Suppose now we are in case \ref{2then0} and choose a Witt basis $\{w_{-1},w_1\}$  for $V$ over $E$  attached to the unipotent subgroup $U$ as in the proof of 
Theorem~\ref{kappainduced}. 
The quotient $J/J^1 \simeq P(\L) \cap B^\times / P_1(\L) \cap B^\times$ is then isomorphic to $SL(2,k_F)$ if $E/F$ is ramified, 
to  $U(1,1)(k_E/k_F)$ if $E/F$ is unramified. 

  We have $ U\cap B^\times= \{ \begin{pmatrix} 1& x\\
0&1\end{pmatrix} / x \in F \} $ with respect to the $E$-basis $\{w_{-1},w_1\}$  and 
$$ U\cap B^\times= \{ u(x) =  \begin{pmatrix} 1&0&0&N(\b) x\\ 0&1&x & 0 \\ 0&0&1&0 \\ 0&0&0&1 
 \end{pmatrix} / x \in F \} $$ in the symplectic basis 
$\{w_{-1},\b w_{-1} , \dfrac {\b} {N(\beta)} \dfrac {w_1} {2} , \dfrac {w_1} {2}\}$ of $V$ over $F$.

Now recall that $\l=\k\otimes\s$, for $\s$ some irreducible cuspidal
representation of $J/J^1$.  Note that all
cuspidal representations $\s$ of $U(1,1)(k_E/k_F)$ or $SL(2,k_F)$ 
 are generic, since the maximal unipotent subgroup is abelian and $\s$ cannot contain the trivial character, by Lemma~\ref{nondegenerate}. 
Here we only need the fact that 
  (the inflation of) $\sigma$ restricted to $J \cap U \cap B^\times$ contains some character of $J \cap U \cap B^\times$, which must have the form 
 $  u(x) \mapsto \psi_F(\a x)$ for some $\a$ in $F$. Since $\chi : \begin{pmatrix} 1&a&b&c\\ 0&1&x & y \\ 0&0&1&z \\ 0&0&0&1 
 \end{pmatrix} \mapsto  \psi_F(\a x) $ is a character of $U$, we see that $\sigma|{J \cap U} $ contains $\chi|{J \cap U} $. 
Since $\k|_{J\cap U}$ contains $\psi_\b$ (Theorem~\ref{kappainduced}), we deduce that
$\l|_{J\cap U}$ contains the restriction to $J \cap U$ of the character $\chi \psi_\b$ of $U$ and, by
Proposition~\ref{BHprop}, $\pi=\cInd_J^G\l$ is generic (and the character $\chi \psi_\b$ is non-degenerate).
\end{proof}

\subsection{Case \ref{2plus0}}

In case \ref{2plus0}, any maximal unipotent subgroup $U$ of $G$ on which  $\psi_\b$ defines a character 
can be written as upper triangular unipotent matrices in a symplectic basis $\{e_{-2},e_{-1}, e_1,  e_2\}$ as in the beginning of \S \ref{other}. 
Here $\{e_{-2}, e_2\}$ is a 
  symplectic basis   of $V^2$  hence $ U \cap B^\times$ is made of matrices of the form 
  $\left(\begin{smallmatrix} 1&0&0&   x\\ 0&1&0 & 0 \\ 0&0&1&0 \\ 0&0&0&1 
 \end{smallmatrix}\right)$. Even though $\sigma$, as a cuspidal hence generic representation of $SL(2,k_F)$, does contain a non trivial character of $J \cap U \cap B^\times$, this character   is by no means the restriction of a character of $U$. This is an heuristic explanation of the fact that   none of the case \ref{2plus0} supercuspidal representations are generic, 
 unfortunately not a proof.   In this section, we prove a crucial result towards non-genericity.  The last step will be taken in section  \ref{S6.degenerate}. 

\begin{Proposition}\label{2plus0nongeneric}
Let $\pi=\cInd_J^G\l$ be a positive level supercuspidal representation from case \ref{2plus0} and let $U$ be a
maximal unipotent subgroup of $G$ such that $\psi_\b|_{U_{der}}=1$.   
 Then the restriction of $\l$ to $J \cap U \cap B^\times$ is a sum of non-trivial characters.  
\end{Proposition}

\begin{proof}
We retain all the notation of section~\ref{Sgen.1} and write $U$ as in the beginning of this section. 
We also write  $ U^2 = U \cap B^\times = U \cap \ \Sp(V^2)$. 
Since $J = (J \cap P(\L) \cap B^\times) J^1$ we have $J \cap U = (J \cap U^2) (J^1 \cap U)$ hence: 
$ \tJ^1 = (J \cap U^2) J ^1$ and   $ \tH^1 = (J \cap U^2) (J ^1 \cap U) H^1$. We claim that
\begin{enumerate}
	\item $J^1 \cap U^2 = H^1 \cap U^2$; 
	\item $\Ind_{\tH^1}^{\tJ^1}\Th_\b \simeq 1_{J \cap U^2} \otimes \eta$. 
\end{enumerate}
The first claim comes from the definitions in \cite{S3} on p.131. Indeed $J^1 \cap \  \Sp(V^2) =  H^1 \cap \ \Sp(V^2) = P_1(\L^2)$. 
For the second, we note first that $1_{J \cap U^2} \otimes \eta$ does define a representation of $\tJ^1$ since 
$J \cap U^2$ normalizes $(J^1, \eta)$ and $\eta|{J^1 \cap U^2}= \eta|{H^1 \cap U^2}$ is a multiple of $ \psi_\b $, trivial on $U^2$
($\b_2=0$). Frobenius reciprocity gives a non-zero intertwining operator between those representations, which are irreducible 
(recall from \S \ref{Sgen.1} that $\Ind_{\tH^1}^{\tJ^1}\Th_\b|{J^1}  =\eta$).

We have $\l=\k\otimes\s$, with $\s$ 
an irreducible cuspidal representation of $J/J^1\cong SL_2(k_F)$. 
In particular, $\s$ is generic and its restriction to 
$\tJ^1/J^1$ is a sum of nondegenerate characters $\chi$. On the other hand, 
the second claim above, plus Theorem \ref{kappainduced}, tell us that 
$\kappa $ is trivial on $J \cap U^2$, q.e.d. 
 \end{proof}

\setcounter{section}{4}

%%%%%%%%%%%%%%%%%%%%%%%%%%%%%%%%%%%%%%%%%%%%%%%%%%%%%%%
\Section{Non-generic representations}\label{S6}
%%%%%%%%%%%%%%%%%%%%%%%%%%%%%%%%%%%%%%%%%%%%%%%%%%%%%%%

Our goal in this section is to show that the positive level supercuspidal 
representations which are  in the list of Theorem~\ref{nongenericscs} are 
indeed non generic.
For cases (\ref{case4}), (\ref{2then0}) and (\ref{2plus2}) 
we will prove that  whenever there is no maximal unipotent subgroup 
on which the function $\psi_\b$ is a character (see Proposition \ref{Uderprop}), 
an irreducible supercuspidal
representation of $G=\Sp_4(F)$ with underlying skew semisimple stratum
$[\L,n,0,\b]$ is not generic. The main tool here is the criterion given in 
Proposition \ref{criterion}, hence we will work out in some detail the restriction of $\psi_\b$ 
to one-parameter subgroups.  This technique will also provide us with a last piece of argument 
to settle non-genericity in case (\ref{2plus0}).
  
\subsection{The function $\bs\psi_{\bs\b}$ on some one-parameter subgroups}
\label{S6.psi}

In a given symplectic basis  $(e_{-2},   e_{-1}, e_1 ,e_2)$ of $V$, we will 
denote by $U_k$, for $k \in \{ -2, -1, 1, 2\}$, the following root subgroup : 
$$
U_k = \{ 1  + x \, t_k \  ; \   x \in F \} \ \text{ with } \ t_k (e_j) = 
\begin{cases}
e_{-k} &\text{ if }j=k; \\
0 & \text{ otherwise}. 
\end{cases}
$$
It is attached to a long root and has a filtration indexed by  $s \in \BZ$:
$U_k (s)  = \{ 1  + x \, t_k \  : \   x \in \pF^s \}$.

Let $\b$ be an element of $A$  and let $\psi_\b$ be the function on $G$ defined by 
 $\psi_\b (x) = \psi( \tr \b (x-1))$, $x \in G$. 
 For  $k $ in $ \{ -2, -1, 1, 2\}$, let $\epsilon(k)$ be the sign of $k$. 
 
 \begin{Lemma} \label{Lemma6.2.1}
 Fix $g \in G$ and $k $ in $ \{ -2, -1, 1, 2\}$. Let $s \in \BZ$. 
\begin{enumerate}
\item For any $x$ in $F$ we have: $ \ 
\psi_\b \,(1+x \, ^gt_k   ) = \psi\, 
 ( \epsilon(k)\, x \,h(ge_{-k}, \b ge_{-k}))$.
\item The character $ \psi_\b$ of $ ^g U_k  $ is
non trivial on $^g U_k(s)  $ 
if and only if $\ s \le - \vF h(ge_{-k}, \b  ge_{-k})$. \rm
\end{enumerate}
\end{Lemma}

\begin{proof} We keep the usual notation $^g u = gug^{-1}$, 
$u^g = g^{-1} u g$.   We have
  for  $x \in F $:
$$ 
\psi_\b  \,(1+x \,^gt_k   ) = \psi\left( \tr (\b  x \,^g t_k
 ) \right)
=\psi\left( \tr ( x  \b^g \,t_k ) \right)
= \psi ( x  (\b^g)_{k, -k})
$$
while the $(k, -k)$ entry of $ \b^g $ is 
$  
( \b^g)_{k, -k} = \epsilon(k) h( e_{-k},  \b^g e_{-k}) 
= \epsilon(k) h( g e_{-k},  \b   g e_{-k})  
$.
\end{proof}

\subsection{Proof for a  maximal simple stratum} \label{S6.maxsimple}

We work in this paragraph with a maximal simple stratum 
$[\L, n, 0 , \b]$  in $\End_F V$, satisfying the assumption  in  
Theorem  \ref{nongenericscs}  namely, in terms of the form $\d$ (see (\ref{skewform}) and \S \ref{biquad}):
\begin{enumerate}
\item {\it $E=F[\b]$ is a biquadratic extension of $F$   with fixed
points $E_0$ under $x\mapsto \bar x$; in particular
$N_{E/E_0}(E^\times)= F^\times E_0^{\times 2}$. } 
\item {\it  The subset  
$ D(V)= \{\b \d(v,v) \ : \ v \in V, v \ne 0\} $ of $E_0^\times$ 
is the  $ N_{E/E_0}(E^\times)$-coset in $ E_0^\times $
that does not contain the kernel of $ \tr_{E_0/F} $.}
\end{enumerate}

Note that lattice duality with respect to the form $\d$,   
 defined by $L^\# = \{ v\in V : \d(v, L) \subseteq \pE\} $ 
for $L$ an $\mathfrak o_E$-lattice in $V$,  coincides with lattice
duality with respect to 
$h$. We may and do  assume that  the self-dual lattice chain $\L$ satisfies  
$\L(i)^\# = \L(d-i)$, $i \in \BZ$, with $d=0$ or $1$. We also may and do assume that $\L$ 
is strict, i.e. has period $2$.

\begin{Lemma} \label{Lemma6.3.1}
Pick $v_0$ in $V$ such that 
$\L(i) = \pE^i v_0$ for $i \in \BZ$. We have 
$\vE \d(v_0,v_0)= 1-d$ and, 
for $v \in V$ and $s$ in $\BZ$: 
$$
\aligned
&\vF h(v, \b v) = \frac 1{e(E_0/F)} \nu_{E_0} (\b \d(v,v));
\\
&\vF h(v, \b v) = s  \iff v\in \L( s+ 
 \tfrac 1 2 (n+d-1) ) -
\L( s+  \tfrac 1 2 (n+d+1) ).
\endaligned
$$ 
\end{Lemma}

\begin{proof}
We have $ h(v, \b v) = -2 \, \tr_{E_0/F} (\b \d(v,v))$ (\S\ref{biquad}); the first
statement is thus a consequence of the following property: 
\begin{quote}
{\it Let $x$ be an element of the  $  F^\times E_0^{\times 2}$-coset in
$ E_0^\times $  that does not contain the kernel of $ \tr_{E_0/F}
$. Then $\vF \tr_{E_0/F}\,  x = \frac 1{e(E_0/F)} \nu_{E_0} x$.}
\end{quote}
Indeed  we have 
$ 
\vF \tr_{E_0/F} \, x = \frac 1{e(E_0/F)} \left[\nu_{E_0} x 
+ \nu_{E_0}  \left(1+ \frac {\tilde x} x \right)\right]
$  where $x \mapsto \tilde x$ 
is the Galois conjugation of $E_0$ over $F$. Let $x$ and $y$ in
$E_0^\times$ such that $  \dfrac {\tilde x} x$ and $   \dfrac {\tilde
 y} y$ belong to $-1+ \mathfrak p_{E_0}$ and  let $u= x/y$. 
Then $\tilde u / u$ belongs to $1+ \mathfrak p_{E_0}$,  which implies that 
$u$ belongs to   $  F^\times E_0^{\times 2}$ (easy checking according to the
ramification of $E_0$ over $F$).   Hence $x$ and $y$ are in the same $
F^\times E_0^{\times 2}$-coset, that must be the coset containing  the
kernel of $ \tr_{E_0/F} $, q.e.d. 

\medskip 
The second statement is now immediate. We can write $v \in V$ as  
$v= u\, v_0$ with $u \in E$. Then: 
$$\aligned 
\vF h(v, \b v) &= \frac 1{e(E_0/F)} \frac 1{e(E/E_0)} \left[
\vE (\b \d(v_0,v_0) )  + \vE (u \bar u)
\right]
\endaligned
$$
hence $ \vF h(v, \b v) = \tfrac 1 2 (-n+1-d) +  \vE u $, q.e.d. 
\end{proof} 

\medskip
We now fix a
  symplectic basis  $(e_{-2},   e_{-1}, e_1 ,e_2)$   adapted
to $\L$: there are non-decreasing functions $\a_s : \BZ \rightarrow \BZ$ such that 
\begin{equation} \label{adapted}
\L(j)= \bigoplus_{s \in \{ -2, -1, 1, 2\}} \pF^{\a_s(j)} e_s \quad (j\in \BZ). 
\end{equation}

\begin{Proposition} \label{Proposition6.3}
For any $g \in P(\L)$, for any $k \in \{ -2, -1, 1, 2\}$, 
the character $\psi_\b$ of $^g U_k$ 
is non trivial on $^g U_k    \cap   P_{\left[\frac {n} 2\right]
+1}(\L)   $.
\end{Proposition}

\begin{proof}
Note first that $P(\L) $ normalizes $ P_{\left[\frac {n} 2\right]
+1}(\L)   $, so   $^g U_k    \cap   P_{\left[\frac {n} 2\right]
+1}(\L)  $ is equal to $ ^g U_k(s)   $,  where $s \in \BZ$ is defined by 
$ U_k   \cap   P_{\left[\frac {n} 2\right]
+1}(\L)  =    U_k(s)   $. 
Using  
Lemma~\ref{LemmaA.2} for our 
lattice sequence $\L$ of period $2$ we get, for $x\in F$: 
$$
\aligned
1+x \, t_k  \in  P_{\left[\frac {n} 2\right]
+1}(\L)  &\iff 
\forall  j  ,   \ x \, t_k \L(j) \subset \L(j + \left[\tfrac {n} 2\right]
+1) 
\\ &\iff 
\forall  j  ,   \ 
   e_{-k} \in \L (j  - 2  \alpha_k(j)   + \left[\tfrac {n} 2\right]
+1- 2\vF x) 
\\ &\iff \vlambda\left(e_{-k}\right)\ge 
\max \,  \{ j  - 2  \alpha_k(j) ; j \in \BZ\}
  + \left[\tfrac {n} 2\right]
+1- 2\vF x  
\\ &\iff \vlambda\left(e_{-k}\right)\ge 
\vlambda\left(e_{k}\right)
  + \left[\tfrac {n} 2\right]
+1- 2\vF x  
\\ &\iff 2\vF x  \ge  
2 \vlambda\left(e_{k}\right)
  + \left[\tfrac {n} 2\right]
+2-d . 
\endaligned 
$$
Hence $s = \vlambda\left(e_{k}\right) +1 + \left[ \dfrac { [\frac {n} 2 ]-d+1} 2 \right]$; we have to prove 
$\ s \le - \vF h(ge_{-k}, \b g e_{-k})$ (Lemma~\ref{Lemma6.2.1}).
Since
$g $ belongs to $ P(\L) $ we have 
by Lemmas~\ref{Lemma6.3.1} and~\ref{LemmaA.2}: 
$$\vF h(ge_{-k}, \b ge_{-k}) = \vF h(e_{-k}, \b e_{-k})= - \vlambda\left(e_{k}\right)+d - 1 -\frac 1 2 (n+d-1) ).$$
 The
condition we need is thus 
  $  \left[ \dfrac { [\frac {n} 2 ]-d+1} 2 \right] 
\le  \frac 1 2 (n-d-1)  $ which holds for $n\ge 1$ (the right hand side is an integer by Lemma~\ref{Lemma6.3.1}). 
\end{proof} 

We now derive Theorem \ref{nongenericscs}  in this case: the representation $\pi$ is not generic. 
Indeed if it was, we could find, by Proposition \ref{criterion},  some $g \in  P(\L)$ and some 
$k \in \{ -2, -1, 1, 2\}$  
such that $\l$ contains the trivial character of $ ^g U_k \cap J$. In particular 
the restriction of $\l$ to $P_{\left[\frac {n} 2\right]
+1}(\L)$ would contain the trivial character of   $^g U_k    \cap   P_{\left[\frac {n} 2\right]
+1}(\L)   $. Since this restriction is a multiple of $\psi_\b$ this is impossible.

\subsection{Proof for a  maximal semi-simple   or   non-maximal simple stratum}\label{maxss}

\medskip
We will in this section treat simultaneously 
 cases (\ref{2then0}) and (\ref{2plus2}) in Theorem \ref{nongenericscs}, granted that: 
 
 \begin{Lemma} \label{Lemmasplitting} 
 Under the assumption of case (\ref{2then0}) in Theorem \ref{nongenericscs}, 
any splitting $V=V^1\bot V^2$ of $V$ into 
two one-dimensional $E$-vector spaces splits the lattice chain $\L$, that is, 
for any $t \in \BZ$:
$$  \L(t)= \L^1(t) \bot \L^2(t)\quad \text{ with } \quad \L^i(t) = \L(t) \cap V^i, \ i=1,2.$$
 \end{Lemma}
 
 \begin{proof} 
 We know from section \ref{S4} that the given assumption   amounts to the fact that 
 the quadratic form $v \mapsto h(v, \b  v) $ on $V$ has no non trivial isotropic vectors, which implies that 
 the anti-hermitian form $\d$ on $V$ defined by 
  $  h(av,w)=\trEF(a \d(v,w))$ for all $ a \in E$, $v,w \in V$, is anisotropic. 
  From  \cite{Bl}A.2,  the lattice chain underlying $\L$ is thus the unique
 self-dual  $\oE$-lattice chain in $V$. On the other hand self-dual  $\oE$-lattice chains 
 $\L^i$ in $V^i$, $i=1,2$, can be summed into a self-dual  $\oE$-lattice chain in $V$. 
 Unicity implies that $\L$ is obtained in this way.
 \end{proof}

We can now let  $[\L, n, 0 , \b]$, with $\L = \L^1 \bot \L^2$,
 $\b = \b_1+\b_2$ and $n = \max\{ n_1, n_2\} $,  be a maximal
 skew semi-simple stratum or a non-maximal skew simple stratum in $\End_F V$: we 
 place ourselves in the situation of 
 case  (\ref{2then0}) or case  (\ref{2plus2}) 
 in Proposition \ref{Uderprop}, case (\ref{2then0})
  being obtained by letting $E_1=E_2=E$ and $\b_1=\b_2=\b$. In particular 
{\it $E_1=F[\b_1]$ is always isomorphic to $E_2=F[\b_2]$\/} 
  and  one checks easily that the conditions (ii) or (iii) 
  in Proposition \ref{Uderprop} are equivalent to saying that {\it the symplectic form
satisfies: 
$$
\forall v_1 \in V^1,  \ \forall v_2 \in V^2,   \ 
-\, h(v_1, \b v_1)\,  h(v_2, \b v_2) \notin N_{E_i/F}(E_i^\times). 
$$}  

For an homogeneous treatment regardless of the ramification over $F$ of the quadratic 
extensions involved, we  use the conventions explained in section \ref{SA.2}, namely the lattices sequences 
$\L^1$, $\L^2$ and $\L$ are all normalized in such a way that they have period $4$ over $F$ and 
duality given by $d=1$.

\begin{Lemma} \label{vLss}
Let $v_1 \in V^1$, $v_2 \in V^2$ and    $v=v_1+v_2 \in V$. 
Let $\mathbf e  = e(E_i/F)$. 
\begin{enumerate}
\def\labelenumi{{\rm(\arabic{enumi})}}
\def\theenumi{\arabic{enumi}}
\item $\vF h(v, \b v)\, = \min \{  \vF h(v_1, \b_1 v_1),\, 
\vF  h(v_2, \b_2 v_2) \}$;  \label{vLss(1)}
\item $\vlambda(v) \ge 2  \vF h(v, \b v)  -  \max \{\dfrac {2\vEone \b_1}{\mathbf e}, 
\dfrac {2\vEtwo \b_2}{\mathbf e}\}$. 
\label{vLss(2)}
\end{enumerate}
\end{Lemma}

\begin{proof} 
\eqref{vLss(1)} We have $  h(v, \b v)  =   h(v_1, \b_1 v_1) + 
h(v_2, \b_2 v_2) $. The assertion follows if the valuations of 
$h(v_i, \b_i v_i)$, $i=1,2$,  are distinct. If they are equal and
finite, write $  h(v, \b v)  =   h(v_1, \b_1 v_1) (1+   
  \frac {h(v_2, \b_2 v_2)}{h(v_1, \b_1 v_1)} )$. Since 
$\frac {h(v_2, \b_2 v_2)}{h(v_1, \b_1 v_1)}$ cannot be congruent
to $-1$ mod $\pF$ (its opposite would be a square hence a norm) the
result follows. 

\eqref{vLss(2)} From Lemma~\ref{LemmaA.4} and (1) we get 
$   \vlambda(v)= \min_{i=1,2}  \vlambdai(v_i)=
\min_{i=1,2}(2 \vF h(v_i, \b_i v_i) - \dfrac {2 \vEi \b_i}{\mathbf e})
$ whence the result.
\end{proof}

We fix symplectic bases  $\{e_i, e_{-i}\}$ of $V^i$, $i=1,2$,   adapted
to $\L_i$ and use notation \ref{adapted}. 
We denote by $\boI^1$ and $\boI^2$ the orthogonal
projections of $V$ onto $V^1$ and $V^2$ respectively. 

We need information on the intersection of one-parameter subgroups  
 $^g U_k $ with  subgroups of $G$ of 
the following form  (with $a_1, a_2 $   positive integers  and  
$a=\max\{a_1, a_2\}$): 
$$
L = \left( 1 + \fa_{a_1}(\L^1) + \fa_{a_2}(\L^2)
+ \fa_{a}(\L) \right) \cap G  .
$$

\begin{Lemma} \label{LemmagUk2}
Let $g \in P(\L)$;  let $g_i = \vlambda\left(\boI^i\left( g e_{-k}\right)\right)\in \BZ \cup +\infty$.
Then   
$$
^g U_k    \cap L =  { }^gU_k(s)   \ \text{ with } \ 
 s = \left[       \dfrac {\vlambda\left(e_k\right) + \max_{i=1,2  } \{
a_i -  g_i  \}+3}{4}
\right] . 
$$
\end{Lemma}

 \begin{proof}
     We have for $x \in F $ and $g \in G$ (see
\cite{BK2} 2.9):
$$
\aligned
1+x \, ^g t_k    \in L  &\iff x \, ^g t_k   \L(j) \subseteq \L^1(j+a_1 ) + 
\L^2(j+a_2 ) \text{ for any }j \in \BZ .
\endaligned
$$
Since $t_k  \, \L(j) = \pF^{\alpha_k(j)} e_{-k}$  we have, if  $g $ belongs to $ P(\L)$:  
$$
\aligned
1+x \, ^gt_k   \in L  &\iff 
\forall  j  ,   \ 
x g e_{-k} \in \L^1(j+a_1  \! - 4  \alpha_k(j) ) + 
\L^2(j+a_2 \!- 4 \alpha_k(j)) \\
&\iff \text{ for } i=1,2, \ \forall  j  ,   \   
4 \vF x +g_i \ge  j   - 4 \alpha_k(j) + a_i  .
\endaligned 
$$
We conclude with   Lemma~\ref{LemmaA.2}.
\end{proof}

We now apply this lemma  to the subgroup $L$ obtained with  
$a_i=\left[\frac {n_i} 2\right] +1$, $i=1,2$. 
Define integers $l_1$ and $l_2$ by 
$\, l_i = \vF h(\boI^i\left( g e_{-k} \right), \b_i \boI^i\left( g
e_{-k} \right))$. 
From Lemmas~\ref{Lemma6.2.1} and~\ref{LemmagUk2}, the character 
$\psi_\b$ of $g U_k g^{-1}$ is non trivial on $g U_k  g^{-1} \cap L $ 
if and only if 
$$
\left[       \dfrac {\vlambda\left(e_k\right) + \max_{i=1,2  } \{
\left[\frac {n_i} 2\right] +1 -  g_i  \}+3}{4} \right]
\le   -\min  \{ l_1, l_2\} . 
$$
On the other hand we
 have  $g_i = 2l_i -  \dfrac {2 \vEi \b_i} {\mathbf e} $  by Lemma
\ref{LemmaA.4},  and Lemma \ref{LemmaA.3} relates $n_i=-\vlambdai(\b_i)$ and $\vEi \b_i$; 
 in any case, one checks: 
$$\ \max_{i=1,2  } \{
\left[\frac {n_i} 2\right] +1 -  g_i  \} = \e -2 \min  \{ l_1, l_2\} \quad  
\text{ with } \e= 0 \text{ or } 1.$$  
It follows that    $\psi_\b$  is
non trivial on $g U_k  g^{-1} \cap L $  if and only if 
$\   \vlambda\left(e_k\right)  
\le   -2\min  \{ l_1, l_2\}-\e $, that is,  
  $\   \vlambda\left(e_k\right)  
\le   -2\vF h(  g e_{-k} , \b  g
e_{-k}  )-\e $. 
Now  Lemma \ref{vLss} gives us: 
$$ \vlambda\left(g e_{-k}\right) \ge 
2\vF h(  g e_{-k} , \b  g
e_{-k}  ) -  \max \{\dfrac {2\vEone \b_1}{\mathbf e}, 
\dfrac {2\vEtwo \b_2}{\mathbf e}\}.$$ 
Since $g$ belongs to $P(\L)$, we have 
$\vlambda\left(g e_{-k}\right) = \vlambda\left( e_{-k}\right)
= - \vlambda\left(  e_{k}\right)$ (Lemma \ref{LemmaA.2}), which implies 
the desired  inequality.  We have just proven: 

\begin{Proposition} \label{Proposition6.4}
{\it For any $g \in P(\L)$, for any $k \in \{ -2, -1, 1, 2\}$, 
the character $\psi_\b$ of $g U_k
g^{-1}$ is non trivial on $g U_k  g^{-1} \cap L $, where  
$$
L = \left( I + \fa_{\left[\frac {n_1} 2\right] +1}(\L^1) + 
\fa_{\left[\frac {n_2} 2\right] +1}(\L^2)
+ \fa_{\left[\frac {n} 2\right] +1}(\L) \right) \cap G.
$$}
\end{Proposition}

At this point we can derive  Theorem \ref{nongenericscs}  in  cases  (\ref{2then0}) and (\ref{2plus2})
 as in the previous subsection: the representation $\pi$ is not generic. 
Indeed if it was, we could find, by Proposition \ref{criterion},  some $g \in  P(\L)$ and some 
$k \in \{ -2, -1, 1, 2\}$  
such that $\l$ contains the trivial character of $ ^g U_k \cap J$. In particular 
the restriction of $\l$ to the above subgroup $L$  would contain the trivial character of   $^g U_k    \cap  L  $. Since this restriction is a multiple of $\psi_\b$ this is impossible.

\subsection{Non-genericity  in the degenerate case (\ref{2plus0})} \label{S6.degenerate}

We let again $[\L, n, 0 , \b]$, with $\L = \L^1 \oplus \L^2$, $\b =
\b_1+\b_2$ and $n = \max\{ n_1, n_2\} $,  be a   skew
semi-simple stratum in $\End_F V$, but we assume that one of the two
simple strata involved is null (case \ref{2plus0}). Although there does exist 
a maximal unipotent subgroup on which $\psi_\b$ is a character, 
this character is then degenerate (Proposition \ref{Uderprop} 
and Remark \ref{Uderpropdegenerate}). We will show that the corresponding 
supercuspidal representation is non-generic, using Proposition 
\ref{2plus0nongeneric} and the criterion 
  \ref{criterion}.
  
  Criterion  \ref{criterion} involves conjugacy by elements in an Iwahori subgroup. 
  We will   find convenient to use the standard Iwahori subgroup, and  
  to use an Iwahori subgroup normalizing the lattice chain $\L$. These  conditions 
  can both be fulfilled at the possible cost of exchanging  $\L^1$ and $ \L^2$, that is, 
  we have to complicate notations and let $\{1,2\} = \{r,s\}$ with 
$[\L^r, n_r, 0 , \b_r]$ not null and $[\L^s, n_s, 0 , \b_s] =
[\L^s, 0, 0 , 0]$; in particular $n=n_r$.

  Since the proof below is rather technical, 
we first sketch it, assuming $\b_2=0$. In a symplectic basis as in  \S \ref{other},
  $\psi_\b$ does define a character of the upper and lower 
triangular unipotent subgroups, trivial on the long root subgroups $U_{\pm 2}$ 
corresponding to the null stratum but non trivial on the other long root subgroups $U_{\pm 1}$. 
As in the previous case, we have a subgroup $L$ of $J$ on which $\lambda$ 
restricts to a multiple of $\psi_\b$. We will show that for any Iwahori conjugate 
$^g U_{\pm 1}$, $\psi_\b$ is non trivial on $^g U_{\pm 1} \cap L$. Next we will identify   subsets 
$X_2$ and $X_{-2}$ of $I$ such that, for $g \in X_{\pm 2}$, $\psi_\b$ is again non trivial on $^g U_{\pm 2}\cap L$. 
The last step will be to show  that for $g \in I-X_{\pm 2}$, if the representation $\l$ contains the trivial character 
of $^g U_{\pm 2}\cap J$, then  it contains the trivial character 
of $  U_{\pm 2}\cap J$ -- this last possibility being excluded by Proposition \ref{2plus0nongeneric}. 
Hence, by Proposition \ref{criterion}, the representation induced from $\lambda$ cannot be generic.

 We pick a symplectic  basis
$(e_{-2},   e_{-1}, e_1 ,e_2)$ of $V$ adapted to $\Lambda$ and such that 
$P(\L)$ contains the standard Iwahori subgroup  $I$  consisting of matrices 
with entries in $\oF$ which are
upper triangular modulo $\pF$. We normalize lattice sequences $\Lambda_1$ 
and $\Lambda_2$ such that they have period $4$ over $F$ and duality invariant 
$1$ and define $\epsilon = 0 $ if $\L_r$ contains a self-dual lattice, 
$\epsilon = 1$ otherwise (\S\ref{SA.2}).    Let: 
$$
L = \left( I + \fa_{1}(\L^s) + \fa_{\left[\frac {n} 2\right] +1}(\L^r)
+ \fa_{\left[\frac {n} 2\right] +1}(\L) \right) \cap G  . 
$$
This is a subgroup of $H^1$ (from the definition \cite{S3} p.131) 
on which $\theta$ restricts to $\psi_\b$ (\cite{S4}   Lemma 3.15). 

\begin{Lemma} \label{Lemma6.6.4} 
Let $k $ in $ \{ -2, -1, 1, 2\}$.
For $y \in P(\L)$, define  $y_r= \vlambda(\boI^r(y e_{-k}))$, 
$y_s= \vlambda(\boI^s(y e_{-k}))$ and   
$\, l_r = \vF h(\boI^r\left( y e_{-k} \right), \b_r \boI^r\left( y
e_{-k} \right))$. 
The character 
$\psi_\b$ of $y U_k y^{-1}$ is non trivial on $y U_k  y^{-1} \cap
L$ if and only if either  $ \ \boI^r(y e_{-k}) =0 \ $ or 
$ \ \boI^r(y e_{-k}) \ne 0 \ $ and 
$$
\left\{\aligned
&y_r \le \vlambda (e_{-k})   + \left[\tfrac {n} 2\right] + 1 - 2 \epsilon  \\
&y_s \ge - \vlambda (e_{-k}) + 4 \, l_r +1  
\endaligned\right.
$$
\end{Lemma}

\begin{proof}
We have  $  h(v, \b v)  =      h(v_r, \b_r v_r)$ (for $v=v_1+v_2$, $v_i
\in V^i$), so  
with Lemmas~\ref{LemmaA.4} and \ref{LemmaA.3}: 
$$
\vF h(v, \b v) = t \iff  \boI^r(v) \in \L^r( 2t+ \left[\tfrac {n} 2\right] + 1 -   \epsilon) 
  -
\L^r(  2t+ \left[\tfrac {n} 2\right] + 1 -   \epsilon +1).
$$
In particular, if $\boI^r(y e_{-k})= 0$, the character $\psi_\b$ 
is   trivial on $ \ y U_k y^{-1} $ (Lemma \ref{Lemma6.2.1}).

We now assume  $  \boI^r(y e_{-k})\ne 0$ and get from  
 Lemma~\ref{LemmaA.4}: 
 \begin{equation}\label{yrlr} 
 y_r = 2 \,  l_r + \left[\frac {n} 2\right] +1-\epsilon . 
 \end{equation}
 We apply Lemma~\ref{LemmagUk2} to $L$:  
$$
 y U_k  y^{-1} \cap L = y  U_k(t) y^{-1} \text{  with  } 
t = \left[       \dfrac {\vlambda (e_{k}) + \max \{ 1 -  y_s, 
\epsilon - 2 \, l_r  \} +3 }{4}
\right] . 
$$
Using Lemma~\ref{Lemma6.2.1} we conclude that the character 
$\psi_\b$ of $y U_k y^{-1}$ is non trivial on $y U_k  y^{-1} \cap L
$ if and only if $t \le -l_r$ whence the result 
(note that $\vlambda (e_{k}) = - \vlambda (e_{-k})$ by Lemma \ref{LemmaA.2}). 
\end{proof}

\begin{Lemma} \label{Lemma6.6.7}
Let $y \in I$. If   $|k| = r$, or if $|k| = s$
and $\boI^r(y e_{-k}) \notin \L^r(\vlambda (e_{-k}) + \left[\frac {n+1} 2\right] )$,
the character $\psi_\b$ of $y U_k y^{-1}$ is non trivial on $y U_k
y^{-1} \cap L $.
\end{Lemma}

\begin{proof} 
  Since $y$ belongs
to $P(\L)$ we certainly have $\boI^r(y e_{-k} )  \in \L^r(\vlambda (e_{-k})) $, 
$ \boI^s(y e_{-k} )  \in \L^s(\vlambda (e_{-k}))$, and: 
\begin{equation} \label{6.6.5}
\aligned 
\text{ either }&\text{\bf (a) }\ \boI^r(y e_{-k} )  \notin \L^r(\vlambda (e_{-k})+1 ) \\
\text{ or }&\text{\bf (b) }\ \boI^r(y e_{-k} )  \in \L^r(\vlambda (e_{-k})+1 )
\text{ and }   \boI^s(y e_{-k} )  \notin \L^s(\vlambda (e_{-k})+1 ). 
\endaligned
\end{equation}
Assume first {\bf (a)}.
 Then the first condition in Lemma~\ref{Lemma6.6.4} is satisfied and the second 
will hold if 
$- \vlambda (e_{-k}) + 4 \, l_r +1   \le \vlambda (e_{-k})$. 
But we have 
 $y_r = 2 \,  l_r + \left[\frac {n} 2\right] +1-\epsilon$ (\ref{yrlr})
 whence 
 $$4l_r = 2(\vlambda (e_{-k}) - \left[\frac {n} 2\right] -1 + \epsilon) \le 
2 \vlambda (e_{-k})-1.$$ 
 Hence:  
\begin{equation} 
\text{\it if } y \in P(\L)
\text{ \it and } \boI^r(y e_{-k}) \notin \L^r(\vlambda (e_{-k})+1  ),    \ 
 \psi_\b   \ 
\text{\it is non trivial on }  y U_k y^{-1} \cap L.  \label{6.6.6}
\end{equation}

The discussion now will rely on $|k|$. If $|k| = r$ and $y \in I$, then 
{\bf (a)} holds  
 and \eqref{6.6.6} gives the result. If now $y \in I$ and $|k| = s$,  we are left with
case {\bf (b)} in \eqref{6.6.5}; in particular $y_s=\vlambda (e_{-k})$. 
We only have to check that the assumption 
$\boI^r(y e_{-k}) \notin \L^r(\vlambda (e_{-k}) + \left[\frac {n+1} 2\right] )$
implies the two inequalities in \ref{Lemma6.6.4}, which is straightforward granted that, when $\epsilon =1$, 
$n$ is even (\ref{LemmaA.3}).  
\end{proof}

\begin{Lemma} \label{Lemma6.6.8}
Let $|k| = s$ and $y \in I$ such that 
$\boI^r(y e_{-k} )  \in   \L^r(\vlambda (e_{-k}) + \left[\frac {n+1} 2\right])$. 
Define $z = I + Z - \bar Z$ with $Z(e_{-k})= - y_{-k, -k}^{-1}
\boI^r(y e_{-k} )$ and $Z(e_t)=0$ for $t\ne -k$. 
Then $z$ belongs to $P_{\left[\frac {n+1} 2\right] }(\L)$, contained in 
$J \cap I$, and $  \boI^r(z y e_{-k} ) = 0$. 
 \end{Lemma}

\begin{proof}
One checks easily that $z$ belongs to $I$ and $  \boI^r(z y e_{-k} ) =
0$. It remains to show that $Z$ belongs to $\fa_{\left[\frac {n+1}
2\right] }(\L)$. We have,  
   using notation \ref{adaptedagain} and Lemma \ref{LemmaA.2}: 
$$
\aligned
Z \in \fa_{\left[\frac {n+1} 2\right] }(\L) 
&\iff  Z \, \pF^{\alpha_{-k}(t)} e_{-k} \subseteq \L(t+ \left[\tfrac {n+1}
2\right])
 \text{ for any } t \in
\BZ \\
\phantom{\fa_{\left[\frac {n+1} 2\right] }(\L)}
&\iff Z e_{-k} \in \L^r(t-4\alpha_{-k}(t) + \left[\tfrac {n+1} 2\right])
 \text{ for any } t \in
\BZ \\ 
\phantom{\fa_{\left[\frac {n+1} 2\right] }(\L)}
&\iff \boI^r(y e_{-k} ) \in \L^r(
\max_{t \in \BZ} \{t-4\alpha_{-k}(t) \}+ \left[\tfrac {n+1} 2\right]) 
\\ 
\phantom{\fa_{\left[\frac {n+1} 2\right] }(\L)}
&\iff \boI^r(y e_{-k} ) \in \L^r( \vlambda (e_{-k})  + \left[\tfrac {n+1} 2\right]) . 
\endaligned
$$ 
\end{proof}

With this in hand we are ready to conclude: 

\begin{Proposition} \label{Proposition6.6.9}
If the representation $\Ind_J^G \lambda$ has a Whittaker model, 
there exists $k \in \{ -s, s\}$ such that $\lambda $ contains the
trivial character of $  U_k   \cap J = U_k   \cap P(\L^s)$. 
\end{Proposition}

\begin{proof}
Recall Proposition \ref{criterion}: 
{\it if the representation $\Ind_J^G \lambda$ has a Whittaker model, 
there exists $k \in \{ -2, -1, 1, 2\}$ and $y \in I$ such that 
$\lambda $ contains the trivial character of $y U_k  y^{-1} \cap J $}. 
Note that if $(k, y)$ is such a pair, so is  $(k, zyx)$ for any $z \in
J\cap I$ and $x \in I \cap N_G(U_k)$. 
 
Assume we are in this situation  and pick such a pair $(k, y)$. 
Since the restriction of $\lambda$ to $L$ is a multiple of $\psi_\b$, 
Lemma~\ref{Lemma6.6.7} tells us that   we must have $|k| = s$ and  
$\boI^r(y e_{-k} )  \in   \L^r(\vlambda (e_{-k}) + \left[\frac {n+1} 2\right])$. 
The Proposition is then an immediate consequence of the following fact: 
\begin{itemize}
\item[($\ast$)] {\it 
The double class $  (J\cap I ) \, y \, (I \cap N_G(U_k))$ contains an
element of $  J\cap I$, indeed of $I \cap P(\L^s)$.} 
\end{itemize}

\medskip
Let us now prove this fact. Assume first that $k=s=2$ (then $r=1$). 
Using the standard Iwahori decomposition of $y \in I$ and the fact that 
upper triangular matrices normalize $U_2$ we may assume that $y$
is a lower triangular unipotent matrix. Now since  
$  \boI^1(y e_{-k} ) $ belongs to $ \L^r(\vlambda (e_{-k}) + \left[\frac {n+1} 2\right]
 )$ we may change $y$ into $zy$ where $z$ is
defined in Lemma~\ref{Lemma6.6.8}. Note that $z$ is also lower
triangular unipotent, hence so is $zy$. Since $  \boI^1(z y e_{-k} ) =
0$, $zy$ has the following shape: 
$$
zy = \begin{pmatrix} 
1&0&0&0\cr 0&1&0&0\cr 0&b&1&0 \cr c&0&0&1
\end{pmatrix}. 
$$ 
The middle block {\scriptsize$ \begin{pmatrix} 1&0 \\  b&1 
\end{pmatrix}$} centralizes $U_2$ so the double class contains 
{\scriptsize$ \begin{pmatrix}
1&0&0&0 \\ 0&1&0&0\\ 0&0&1&0 \\ c&0&0&1
\end{pmatrix}$}
which belongs to $I \cap\Sp(V^2)$ hence to $P(\L) \cap\Sp(V^2)= P(\L^2)$, 
q.e.d. The case $k=-s=-2$ is identical, replacing lower triangular by
upper triangular. The cases $k=s=1$ and $k=-s=-1$ are obtained
similarly, using conjugation by 
$w=${\scriptsize$\begin{pmatrix} 
0&0&1&0\\ 0&0&0&1\\ \piF&0&0&0 \\ 0&\piF&0&0
\end{pmatrix}$} or $w^{-1}$, elements of $\GSp_4(F)$  normalizing $I$,
to get a convenient Iwahori decomposition for the initial element $y$. 
\end{proof}

 We are at last ready to prove non-genericity of supercuspidal representations 
 coming from case (\ref{2plus0}), which will finish the proof  of Theorem \ref{nongenericscs}. 
 The group $  U_k   \cap J$ above 
 is   equal to $J \cap U \cap B^\times$ for some unipotent subgroup $U$ 
 chosen as in Proposition \ref{2plus0nongeneric} -- indeed we have $k \in \{ -s, s\}$
 so $U_k$ is a long root subgroup attached to the two-dimensional space in which we have a null stratum. 
We know from Proposition \ref{2plus0nongeneric} that the restriction of $\lambda$ to 
$  U_k   \cap J$,  a sum of non trivial characters, cannot contain the trivial character. 
 Hence 
 $\Ind_J^G \lambda$ doesn't have a Whittaker model. 

\setcounter{section}{5}

%%%%%%%%%%%%%%%%%%%%%%%%%%%%%%%%%%%%%%%%%%%%%%%%%%%%%%%
\Section{Appendix: normalization of lattice sequences}\label{SA}
%%%%%%%%%%%%%%%%%%%%%%%%%%%%%%%%%%%%%%%%%%%%%%%%%%%%%%%

We gather in the first two subsections the technical information about direct sums of 
lattice sequences that we need in several parts of the paper. Specifically, 
in most cases we deal with direct sums of 
self-dual lattice sequences in two-dimensional symplectic spaces and it will be 
 convenient to homogeneize  their $F$-periods and duality invariants. 
 
 Then we proceed with preliminary results in view of the proof, in \S \ref{SA.4}, 
 of Proposition~\ref{thetaagreeswithpsi} 

\subsection{Self-dual lattice sequences}\label{SA.1}

Let $\L$ be an $\oF$-lattice sequence in a finite dimensional $F$-vector space $V$ and let 
$\fa_k(\L)$ be the corresponding filtration of $A$. We define:
$$
  \quad  \vlambda(v) = \max \{i \in \BZ ; v \in \L(i) \} \quad (v \in V) ; \quad 
 \vlambda(g) = \max \{i \in \BZ ; g \in \fa_i(\L) \} \quad (g \in A ). 
$$

When $V$ is   equipped with a symplectic form $h$, lattice duality with respect to $h$    
is defined by $L^\# = \{ v\in V : h(v, L) \subseteq \pF\} $ 
for $L$ an $\oF$-lattice in $V$. An $\oF$-lattice sequence $\L$ in $V$ is 
\it self-dual \rm if there is an integer $d(\L)$, the duality invariant 
of $\L$, such that   
 $\L(t)^\# = \L(d(\L) -t)$. 

We now let $\L$ be a self-dual lattice sequence of period $e$ in $V$  and we pick  
a  symplectic basis  $(e_{-2},   e_{-1}, e_1 ,e_2)$   adapted
to $\L$: there are non-decreasing functions $\a_s : \BZ \rightarrow \BZ$ such that 
\begin{equation}\label{adaptedagain}
\L(j)= \bigoplus_{s \in \{ -2, -1, 1, 2\}} \pF^{\a_s(j)} e_s \quad (j\in \BZ). 
\end{equation}
We will need the following straightforward property: 

\begin{Lemma}  \label{LemmaA.2} 
For $k $ in $ \{ -2, -1, 1, 2\}$, we have 
$$ \vlambda\left(e_k\right)= - \vlambda\left(e_{-k}\right)+d(\L) - 1 = \max
  \,  \{ j - e  \alpha_k(j)  ; j \in \BZ\} .$$ 
\end{Lemma}

\begin{proof} Since the period of $\L$ is $e$, the two sets $\{ j - e  \alpha_k(j)  ; j \in \BZ\}$ 
and $\{i \in \BZ ;    \alpha_k(i) =0 \}$ coincide:  
the valuation of $e_k$ is the maximum of either one. Next 
we use duality to check that 
$$ 
e_{-k} \notin \L(x+d ) \iff e_{-k} \notin \L(-x)^\# 
\iff h(e_{-k}, \pF^{\alpha_k(-x)}e_k) \notin \pF 
\iff 
\alpha_k(-x) \le 0 . 
$$ 
\end{proof}

\subsection{Normalization of our  lattice sequences}\label{SA.2} 

We start as in \S\ref{S3} case \eqref{2plus2} with an orthogonal
decomposition $V = V^1 \perp V^2 $ and not null skew simple strata
$[\L^i, n_i, 0 , \b_i]$ in  $\End_F( V^i)$. We let
$E_i=F[\b_i]$. We will find convenient to normalize the lattice sequences 
$\L^i$ in such a way that their sum $\L$ is given by 
$\L(t)= \L^1(t) \bot \L^2(t)$ for any $t \in \BZ$, and is self-dual. 
This will be the case provided that $\L^1$ and $\L^2$ have the same period 
and $d(\L^1) = d(\L^2) = 1$ (see \cite{S4} for instance). 

The quadratic form $v \mapsto h(v, \b_i  v) $ on $V^i$ has no non trivial isotropic vectors, since 
 the anti-hermitian form $\d_i$ on $V^i$ defined by 
  $  h(av,w)=\trEiF(a \d_i(v,w))$ for all $ a \in E_i$, $v,w \in V^i$, is anisotropic. Let $v_i$ be a basis of $V^i $ over $E_i$. There is 
  some $u_i$ in $E_i$, satisfying $\bar u_i = -u_i$,  such that this form reads 
  $\d_i(xv_i, yv_i)= ux\bar y$ for all $x, y \in E_i$. Since $E_i$ is tame over $F$, lattice duality is the same for $h$ and $\d_i$, namely
  $\L^i(t)^\# = \{ v \in V^i ; \d_i(v , \L^i(t))\subset \pEi \}$. 
  We have here $(\pEi^t v_i )^\#  = \pEi^{d_i-t} v_i$ with $d_i= 1-\vE u_i$. 
  Hence, up to a translation in indices, the unique self-dual $\mathfrak o_{E_i}$-lattice chain $(L_i)_{i \in \mathbb Z}$ in $V^i$  satisfies one of three possibilities: 
  \begin{enumerate}
  \item $E_i$ is ramified over $F$ and $d_i = 0$;
  \item $E_i$ is unramified over $F$ and $d_i=0$;
    \item $E_i$ is unramified over $F$ and $d_i=1$.
  \end{enumerate}
  
 In the two first cases we put $L_i^\prime (t) = L_i ([\frac t 2 ])$ 
 and get a lattice sequence with duality invariant $d_i^ \prime =1$; in the third case we keep $L_i^\prime = L_i$. In the ramified case we put 
 $\L_i = L_i^\prime$: it has period $4$. We now need to put 
 $\L_i = 2 L_i^\prime $ in the second case, 
 $\L_i = 4 L_i^\prime $ in the third  case, 
 and we get a normalization of our $\oEi$- lattice sequences in $V_i$ 
 such that their period is $4$ and duality invariant $1$. 
 We will use the following straightforward properties of $\L_i$, in each of the three cases above: 
 
 \begin{Lemma}  \label{LemmaA.3} 
Normalize the lattice sequence $\L^i$ such that 
its period over $F$ is $4$ and $d(\L^i)=1$.  
  \begin{enumerate}
  \item If $E_i$ is ramified over $F$ then 
  $\vlambdai(V^i-\{0\})= 2 \BZ +1$  and 
  $\vlambdai(\b_i)=2 \vEi(\b_i) +1$;
  \item If $E_i$ is unramified over $F$ and $\L_i$ contains a self-dual lattice then 
  $\vlambdai(V^i-\{0\})= 4 \BZ +2$  and 
  $\vlambdai(\b_i)=2(2 \vEi(\b_i) +1)$;
   \item If $E_i$ is unramified over $F$ and $\L_i$ does not contain a self-dual lattice then 
  $\vlambdai(V^i-\{0\})= 4 \BZ $  and 
  $\vlambdai(\b_i)=4 \vEi(\b_i)  $.
  \end{enumerate}
\end{Lemma}

  \medskip
  We   need to relate, under these conventions, the valuations relative to the lattice chains $\L_i$ with the valuations over $F$ of the quadratic forms
$h(v, \b_i v)$, $v \in V^i$.

\begin{Lemma}  \label{LemmaA.4} 
Let $\mathbf e_i = e(E_i/F)$ and normalize the lattice sequence $\L^i$ such that 
its period over $F$ is $4$ and $d(\L^i)=1$.  
 For any $v \in V^i$   we have: 
 $$ \vlambdai (v) = 
2 \vF h(v, \b_i v) - \dfrac {2 \vEi \b_i}{\mathbf e_i} . $$ 
\end{Lemma}

\begin{proof}
For $v$ in $V^i$ and $a$ in $E_i$ we have 
$h(av , \b_i av)= {\hbox{\rm N}}_{E_i/F}(a) h(v , \b_i v)$. 
The map 
$v \mapsto h(v , \b_i v)$ on $V_i$  is thus constant on the sets 
$\L_i(t) - \L_i(t+1)$ hence factors through the valuation $\vlambdai$: 
there is a map $\phi$, defined on the image of $\vlambdai$ and with values in $\BZ$, such that $\vF h(v, \b_i v) = \phi(\vlambdai(v))$ 
for any non zero $v$ in $V^i$.
Periodicity over $E_i$ implies 
$\phi(t+2\mathbf e_i)= \phi(t)+\mathbf e_i$, $t \in \BZ$: computing one value of $\phi$ is enough. We certainly have 
$\vF h(v, \b_i v)= \dfrac {1}{\mathbf e_i} (\vEi \b_i   + \vEi \d_i(v,v))$ whence: 
 \begin{enumerate}
  \item if $E_i$ is ramified over $F$, then $\L_i(1)=\L_i(0)=\L_i(1)^\#$ so   $\vlambdai(v)=1$ implies  $\vEi \d_i(v,v))=1$ and 
$\vF h(v, \b_i v)= \dfrac 1 2 ( \vEi \b_i   + 1)$;  
  \item if $E_i$ is unramified over $F$ and $\L_i$ contains a self-dual lattice then $\L_i(2)^\#=\L_i(-1)= \L_i(2)$ so   $\vlambdai(v)=2$ implies  $\vEi \d_i(v,v))=1$ and 
$\vF h(v, \b_i v)= \vEi \b_i   + 1 $;
   \item If $E_i$ is unramified over $F$ and $\L_i$ does not contain a self-dual lattice then $\L_i(0)^\#=\L_i(1)= \piF \L_i(0)$ 
   so   $\vlambdai(v)=0$ implies  $\vEi \d_i(v,v))=0$ and 
$\vF h(v, \b_i v)= \vEi \b_i     $. 
  \end{enumerate}
\end{proof}

%%%%%%%%%%%%%%%%%%%%%%%%%%%%%%%%%%%%%%%%%%%%%%%%%%%%%%%
\subsection{Some intersections}\label{SA.3} 
%%%%%%%%%%%%%%%%%%%%%%%%%%%%%%%%%%%%%%%%%%%%%%%%%%%%%%%

The following Lemma and Corollary were suggested by  
Vytautas Paskunas. In the proofs, we may (and often will) ignore the
condition that the stratum and character be skew, since the results
in the skew case follow immediately by restriction to $G$ -- that is,
we actually prove the statements for $\s$-stable semisimple strata
in $\GL_4$. To ease notation, this will be implicit so that, in the
proofs below, $U$ should be a $\s$-stable maximal unipotent subgroup of
$\GL_4$, etc.

\medskip

First we need some notation in the semisimple
case. For $[\L,n,0,\b]$ a semisimple stratum in $V$ with splitting
$V=\bigoplus_{i=1}^2 V^i$, we write $A=\oplus_{i,j=1}^2 A^{ij}$ in
block notation, where $A^{ij}=\Hom_F(V^j,V^i)=\boI^j A \boI^i$ and
$\boI^i$ is the projection onto $V^i$ with kernel $V^{3-i}$. For
$\bk=(k_1,k_2)\in\BZ^2$, with $k_1\ge k_2\ge 1$, define
$$
\fa_\bk=\fa_\bk(\L) = \begin{pmatrix} \fa^{11}_{k_1} & \fa^{12}_{k_1}
  \\ \fa^{21}_{k_1} & \fa^{22}_{k_2} \end{pmatrix}
$$
and $\U^\bk(\L)=1+\fa_\bk$.

\begin{Lemma}\label{maxdecomp} Let $[\L,n,0,\b]$ be a semisimple stratum in
$A$, with splitting $V=\bigoplus_{i=1}^l V^i$, $1\le l\le 2$. Suppose
also that \emph{$F[\b]$ is of maximal degree over $F$}, and let $U$ be a
maximal unipotent subgroup of $G$. Write $B$ for the centralizer in $A$ of
$\b$. For $k\ge m\ge 1$, we have 
$$
\left((\U^m(\L)\cap B)\U^k(\L)\right)\cap U = (\U^k(\L)\cap U).
$$ 
\end{Lemma}

We remark that the proof below is just for the case which
interests us here (so that there are at most $2$ pieces in the splitting)
but it is straightforward to generalize the Lemma to the case where
there are an arbitrary number of pieces.

\begin{proof} We will prove the corresponding additive statement.
Writing $U=1+\BN$, $\fb_m = \fa_m (\L) \cap B$, it is:
$$
(\fb_m+\fa_k)\cap\BN = \fa_k\cap\BN.
$$
Notice that, since $B=F[\b]$, the lattice $\fb_m$ contains no
non-trivial nilpotent elements. We will show that, for $1\le m< k$,
$$
(\fb_{m}+\fa_k)\cap\BN \subset (\fb_{m+1}+\fa_k)\cap\BN
$$
and the result follows at once by an easy induction. So
suppose $\epsilon\in\fb_m$ is such that $(\epsilon+\fa_{k})\cap\BN
\ne \emptyset$. In particular, $(\epsilon+\fa_{m+1})\cap\BN
\ne \emptyset$ so there exists $s>0$ such that
$\epsilon^s\in\fa_{sm+1}$. But then $\epsilon^s\in\fb_{sm+1}$ so (by
\cite{Bu}, working block-by-block), the coset $\epsilon+\fb_{m+1}$
contains a nilpotent element, which must be $0$. We deduce that
$\epsilon\in\fb_{m+1}$, as required.
\end{proof}

For $[\L,n,0,\b]$ a semisimple stratum in $V$ with splitting
$V=\bigoplus_{i=1}^2 V^i$, we write
$M_{sp}=\Aut_F(V^1)\times\Aut_F(V^2)$, $U_{sp}=1+A^{12}$, 
$P_{sp}=M_{sp}U_{sp}$ and $\bar U_{sp}=1+A^{21}$. 

\begin{Corollary}\label{decomp} Let $[\L,n,0,\b]$ be a semisimple stratum in
$A$, with splitting $V=\bigoplus_{i=1}^l V^i$, $1\le l\le 2$, and let
$U$ be a maximal unipotent subgroup of $G$. Write $B$ for the
centralizer in $A$ of $\b$. We suppose that
\begin{enumerate}
\item $U$ has an Iwahori decomposition with respect to $(M_{sp},P_{sp})$;
\item $U\cap B^\times$ is a maximal unipotent subgroup of $B^\times$.
\end{enumerate}
Then, for $k_1\ge \cdots\ge k_l\ge m\ge 1$ and $\bk=(k_1,\dots,k_l)$,
we have
\begin{equation}\label{multequation}
\left((\U^m(\L)\cap B)\U^\bk(\L)\right)\cap U = (\U^m(\L)\cap B\cap U)
(\U^\bk(\L)\cap U).
\end{equation}
\end{Corollary}

Again the statement is easily generalized to the case when the
splitting has more than two pieces.
Note also that, in the simple case, condition (i) is empty while
condition (ii) is implied, for example, by the condition
$\psi_\b|_{U_{der}}=1$ (\cite{BH} Proposition 2.2). This is not true for semisimple strata.

\begin{proof} First we reduce to the simple case. We notice that
$\U^\bk(\L)$ and $B^\times \subset M_{sp}$ have Iwahori decompositions with
respect to $(M_{sp},P_{sp})$. Since $U$ also has such a decomposition, we
reduce to proving that we have equality in \eqref{multequation} when we
intersect both sides with $U_{sp}$, with $M_{sp}$ and with $\bar U_{sp}$
respectively. Since $B\subset M_{sp}$, this is immediate for the
unipotent radicals $\bar U_{sp}$, $U_{sp}$. Hence we are reduced to the
intersection with $M_{sp}$, which, block-by-block, is just the simple case.

So now suppose $[\L,n,0,\b]$ is a simple stratum. As in Lemma \ref{maxdecomp},
we will prove the corresponding additive statement: 
writing $U=1+\BN$, it is:
\begin{equation}\label{adddecomp}
(\fb_m+\fa_\bk)\cap\BN = (\fb_m\cap\BN) + (\fa_\bk\cap\BN),
\end{equation}
where $k=k_1$. We will reduce to the case where $E=F[\b]$ is maximal
and invoke Lemma \ref{maxdecomp}. 

Write $d=[E:F]$; then, in the flag corresponding to $U$,
$$
\{0\}=V_0 \subset V_1\subset \cdots\subset V_N=V,
$$
the subspace $V_{di}$ is an $E$-subspace, for $0\le i\le N/d$.
Let $U_0=1+\BN_0$ be the unipotent subgroup corresponding to the
maximal $E$-flag 
$$
\{0\}=V_0\subset V_d\subset \cdots\subset V_{di}\subset\cdots\subset V_N=V
$$
and let $P_0=1+\BP_0$ the corresponding parabolic subgroup. 
There exists an  $E$-decomposition   $V= \oplus_{i=1}^{N/d} \, W_i$ of $V$ 
such that  for $0\le i\le N/d$,  $V_{di}= \oplus_{j=1}^{i} W_j$ 
and such that for every $t \in \mathbb Z$, 
$\L(t) = \oplus_{i=1}^{N/d} \, \L(t) \cap W_i$ (as a suitable variant of, 
e.g., \cite{W} \S II.1).  
Let $L_0$ be the corresponding Levi component of $P_0$    and let $\bar U_0=1+\bar{\BN}_0$ be the
unipotent subgroup opposite $U_0$ with respect to $L_0$. 

The lattices $\fb_m$ and $\fa_k$ have (additive) Iwahori
decompositions with respect to $\bar{\BN}_0,\BP_0$ (\cite{BH0} \S 10) so  
$$
(\fb_m+\fa_k)\cap \BP_0 = 
\fb_m\cap\BL_0 + \fb_m\cap\BN_0 + 
\fa_k\cap\BL_0+ \fa_k\cap\BN_0
$$
and we have
$$
(\fb_m+\fa_k)\cap \BN = 
(\fb_m+\fa_k)\cap \BP_0\cap\BN = 
(\fb_m\cap\BL_0 + \fa_k\cap\BL_0)\cap\BN +
\fb_m\cap\BN_0 + \fa_k\cap\BN_0
$$
Hence we are reduced to showing 
$$
(\fb_m\cap\BL_0 + \fa_k\cap\BL_0)\cap\BN = 
\fb_m\cap\BL_0\cap\BN + \fa_k\cap\BL_0\cap\BN,
$$
which is the same as \eqref{adddecomp} in $\BL_0$, where we can work
block-by-block. In each block, the field extension $E$ is maximal, so
we have indeed reduced to the maximal case, which is Lemma \ref{maxdecomp}. 
\end{proof}

Finally, we will need one more similar result, in the case of a
minimal semisimple stratum -- in this case the extra conditions of the
previous  Corollary are not satisfied.

\begin{Lemma}\label{ssdecomp} 
Let $[\L,n,0,\b]$ be a skew semisimple stratum, with splitting
$V=V^1\bigoplus V^2$, $1\le l\le 2$. Suppose that $\b_i=\boI^1\b\boI^i$ is
minimal, for $i=1,2$ and that $[\L,n,n-1,\b]$ is \emph{not} equivalent
to a simple stratum. Assume also (without loss of generality) that $n_1\ge n_2\ge 1$. 
Put $k_i=[\frac{n_i}2]+1$, for $i=1,2$, and
$\bk=(k_1,k_2)$. Then $H^1=(\U^1(\L)\cap B)\U^\bk(\L)$ and, if $U$ is 
a maximal unipotent subgroup
  of $G$ on which $\psi_\b$ defines a character:  
$$
H^1\cap U=\U^\bk(\L)\cap U.
$$
\end{Lemma}

\begin{proof} We work under the conventions made in the previous subsection. The
 subgroup $U$
is necessarily given (see Proposition \ref{Uderprop}) by a flag of the
following form: 
$$
\{0\} \subset \la v_1+v_2 \ra \subset \la v_1+v_2, v_{-1}-v_{-2} \ra 
\subset \la v_1,v_2,v_{-1}-v_{-2}\ra \subset V, 
$$
where $v_i\in V^i$ are non zero vectors  such that
$ 
h(v_1,\b_1v_1)=-h(v_2,\b_2v_2) =\mu 
$  
and  $v_{-1}=\mu^{-1}\b_1 v_1$, $v_{-2}=-\mu^{-1}\b_2 v_2$. In
particular $\{v_i,v_{-i}\}$ is a symplectic basis for $V^i$, adapted to 
$\L^i$. 

\medskip
That $H^1=(\U^1(\L)\cap B)\U^\bk(\L)$ is clear from the
definition. For the intersection property, we will need:

\begin{Lemma}  \label{LemmaA.5} 
 Put $\nu=\vlambdaone(v_1)- \vlambdatwo(v_2)$ and let $E_{i,j}$ be the linear map sending $v_j$ to $v_i$ and all other basis elements to $0$. 
 We have 
 \begin{equation}\label{size}
\vlambda(E_{1,2})=  \nu,\qquad \vlambda(E_{-2,-1}) =n_1-n_2-\nu.
\end{equation}
Furthermore, if $n_1 \ge n_2$, then 
$ 0 \le \nu \le n_1-n_2$.
\end{Lemma}

\begin{proof}
Computing the valuations of $E_{1,2}$ and $E_{-2,-1}$ is simple
checking. Lemma \ref{LemmaA.4}  
combined with our asssumption on $v_1$, $v_2$ implies $\nu =  - \dfrac
{2 \vEone \b_1}{\mathbf e_1} 
 + \dfrac {2 \vEtwo \b_2}{\mathbf e_2}$. 
The inequality $ 0 \le \nu \le n_1-n_2$ then follows from Lemma
\ref{LemmaA.3}.
\end{proof}

Write $U=1+\BN$. Then the elements of $\BN$ can be
written (as matrices with respect to the basis $\{v_1,v_{-1},v_2,v_{-2}\}$)
$$
x=\begin{pmatrix} 
a&e&-a&c\\ -d&-b&d&-b\\ a&e+f&-a&c+f\\ d&b&-d&b 
\end{pmatrix},
$$
for $a,b,c,d,e,f\in F$. 

Now suppose that $x$ as above also lies in the lattice (in block
matrix form, each block $2\times 2$)
$$
\begin{pmatrix} \fa_{m} & \fa_{k_1} \\ \fa_{k_1} & \fa_{k_2} \end{pmatrix}, 
$$
for some $m < k_1$. Then, using~\eqref{size}, we get:
\begin{itemize}
\item $aE_{2,1}\in\fa_{k_1}$ so $aE_{1,1}=aE_{1,2} E_{2,1} \in\fa_{k_1+\nu}$;
\item $bE_{-1,-2}\in\fa_{k_1}$ so $bE_{-1,-1}=bE_{-1,-2}E_{-2,-1}
\in\fa_{k_1+n_1-n_2-\nu}$; 
\item $(c+f)E_{2,-2}\in\fa_{k_2}$ so
$(c+f)E_{1,-2}=(c+f)E_{1,2}E_{2,-2}\in\fa_{k_2+\nu}$;  

$cE_{1,-2}\in\fa_{k_1}$  so $fE_{1,-2}\in\fa_{\text{min}\{k_1 , k_2+\nu\}}$ 

and $fE_{1,-1}=fE_{1,-2}E_{-2,-1}\in\fa_{\text{min}\{k_1+n_1-n_2-\nu,
k_2+n_1-n_2\}}$. 

$(e+f)E_{2,-1}\in\fa_{k_1}$ so
$(e+f)E_{1,-1}=(e+f)E_{1,2}E_{2,-1}\in\fa_{k_1+\nu}$. 
\end{itemize}
Writing $t=\min\{\nu,n_1-n_2-\nu, k_2+n_1-n_2-k_1\}\ge 0$, we have
$aE_{1,1},bE_{-1,-1},eE_{1,-1}\in\fa_{k_1+t}\subset\fa_{m+1}$. In
particular, looking at the ``top-left'' block of $x$, we have 
$$
\begin{pmatrix} a & e \\ -d & -b \end{pmatrix}
\in 
\begin{pmatrix} 0&0\\ -d&0 \end{pmatrix} + \fa_{m+1}.
$$

Now we prove that, for $m\le k_2$ and $\bok=(k_1,k_2)$, 
$$
\left((\U^m(\L)\cap B)\U^\bok(\L)\right)\cap U = \U^{\bok}(\L)\cap U.
$$
By Lemma~\ref{maxdecomp}, we have
$$
\left((\U^m(\L)\cap B)\U^\bok(\L)\right)\cap U \subset 
\left((\U^m(\L)\cap B)\U^{k_2}(\L)\right)\cap U = \U^{k_2}(\L)\cap U.
$$
Then we need only prove, for $k_2\le m<k_1$ (the additive statement)
$$
\begin{pmatrix} \fb_m+\fa_{k_1} & \fa_{k_1} \\ \fa_{k_1} & \fa_{k_2}
\end{pmatrix} \cap \BN \subset 
\begin{pmatrix}  \fb_{m+1}+\fa_{k_1} & \fa_{k_1} \\ \fa_{k_1} &
\fa_{k_2}\end{pmatrix}.
$$
But we have seen above that, if
$\epsilon\in\begin{pmatrix}\fb_m&0\\0&0\end{pmatrix}$ is such that
$\epsilon+\fa_{\bok}$ contains an element of $\BN$, then
$\epsilon+\begin{pmatrix}\fa_{m+1}&0\\0&0\end{pmatrix}$ also contains
a nilpotent element. But then (as in the proof of
Lemma~\ref{maxdecomp})
$\epsilon+\begin{pmatrix}\fb_{m+1}&0\\0&0\end{pmatrix}$ also contains
a nilpotent element, which must be $0$. Hence $\epsilon\in\fb_{m+1}$,
as required.
\end{proof}

%%%%%%%%%%%%%%%%%%%%%%%%%%%%%%%%%%%%%%%%%%%%%%%%%%%%%%%
\subsection{Proof of Proposition~\ref{thetaagreeswithpsi}}\label{SA.4}
%%%%%%%%%%%%%%%%%%%%%%%%%%%%%%%%%%%%%%%%%%%%%%%%%%%%%%%

It remains to prove Proposition~\ref{thetaagreeswithpsi}: If
$\th\in\Clob$ is a skew semisimple character and $U$ is a 
maximal unipotent subgroup of $G$ such that $\psi_\b|_{U_{der}}=1$, then
$$
\th|_{H^1\cap U} = \psi_\b|_{H^1\cap U}.
$$
As in the previous section, we actually prove the same statement for
$\s$-stable semisimple strata in $\GL_4$, and the result follows by
restriction to $G$. We remark that a skew element $\b$ 
generating a field such that $[F[\b]:F]=2$ is necessarily 
minimal. We proceed on a case-by-case basis:

{\bf Simple case} (Cases (I) and (II))
We proceed by induction on $r=-k_0(\b,\L)$. 
We make the following additional hypothesis:
\begin{itemize}
\item[($*$)] there exists a simple stratum $[\L,n,r,\g]$ equivalent
to $[\L,n,r,\b]$ such that $\psi_\g|_{U_{der}}=1$.
\end{itemize}
In particular, we can then use the inductive hypothesis for the
stratum $[\L,n,0,\g]$ with the same unipotent subgroup $U$. Note that
this hypothesis is certainly satisfied when $\b$ is minimal, since we
can take $\g=0$. We will show later that it is also satisfied in the
other cases that are relevant to us here.

We have $H^1(\b,\L)=(\U^1(\L)\cap B)H^{[\frac r2]+1}(\b,\L)$ so, by
Corollary~\ref{decomp}, $H^1\cap U=H^{[\frac r2]+1}\cap U$. Moreover,
$H^{[\frac r2]+1}(\b,\L)=H^{[\frac r2]+1}(\g,\L)$ and 
$\th|_{H^{[\frac r2]+1}}=\th_0|_{H^{[\frac r2]+1}} \psi_{\b-\g}$, for 
some $\th_0\in\CC(\L,0,\g)$. But, by the inductive hypothesis,
$\th_0$ agrees with $\psi_\g$ on $H^{[\frac r2]+1}\cap U$ and the
result follows.

To finish, we must show that ($*$) is satisfied. The only case to
consider is when $[\L,n,0,\b]$ is a skew simple stratum with $F[\b]$
maximal (of degree $4$) and $r=-k_0(\b,\L)<n$. Then $[\L,n,r,\b]$ is
equivalent to some skew simple stratum $[\L,n,r,\g_0]$, with $\g_0$
\emph{minimal} and $F[\g_0]$ of degree $2$ over $F$. Note that, since
$p\ne 2$, all extensions are tame here. Note also that the
flag corresponding to the unipotent subgroup $U$ is given by $V_i=\la
v,\b v,\cdots, \b^{i-1} v\ra$, for some $v\in V$ with $h(v,\b
v)=0$. Also put $\zeta=h(\b v,\b^2 v)\ne 0$, and $d=-v_F(\zeta)$.

Let $P(X)=X^2+\l\in F[X]$ be the minimal polynomial of $\g_0$ and put
$e_{-2}=v$, $e_{-1}=\b v$, $e_1=\zeta^{-1}P(\b) v$, $e_2=-\b
\zeta^{-1}P(\b)v-k\b v$, where $k=\zeta^{-2} h(\b v,\b^2 v)-2\l\zeta^{-1}$
(this is a symplectic basis). With respect to this basis, $\b$ has
matrix
$$
\begin{pmatrix} 0&-\l&0&\mu\\ 1&0&k&0 \\ 0&\zeta&0&\l \\ 0&0&-1&0 
\end{pmatrix}
$$
for some $\mu\in F$. For $i,j\in\{\pm 1,\pm 2\}$, write $E_{i,j}$ for
the linear map sending $e_j$ to $e_i$ and all other basis vectors to $0$.
We then have
$$
E_{-2,-1}\in\fa_{n},\qquad E_{1,-1}\in\fa_{de-r}.
$$

Write $\b=\g_0+c_0$, with $c_0\in\fa_{-r}$. Then $\b^2+\l=c_0\g_0+\g_0c_0+c_0^2\in\fa_{-n-r}$. From the matrix description of $\b$, we get that
$$
\b^2+\l=\begin{pmatrix} 0&0&-(\mu+k\l)&0\\ 0&k\zeta&0&(\mu+k\l) \\ 
\zeta&0&k\zeta&0 \\ 0&-\zeta&0&0 
\end{pmatrix}
$$
Then:
\begin{itemize}
\item $(\mu+k\l)E_{-1,2}\in\fa_{-n-r}$ and
  $(\mu+k\l)E_{-2,2}=(\mu+k\l)E_{-2,-1}E_{-1,2}\in\fa_{-r}$; 
\item $\zeta E_{1,-1}\in\fa_{-r}$.
\end{itemize}
Hence 
$$
c=\begin{pmatrix} 0&0&0&\mu+k\l\\ 0&0&0&0 \\ 0&\zeta&0&0 \\ 0&0&0&0 
\end{pmatrix} \in \fa_{-r}.
$$
Put $\g=\b-c$. Then $[\L,n,r,\g]$ is a stratum as required by
hypothesis ($*$).

\medskip

{\bf Minimal semisimple case} (Case (III))
Now suppose $[\L,n,0,\b]$ is a skew semisimple stratum, with splitting
$V=V^1\bigoplus V^2$, $1\le l\le 2$, that $\b_i=\boI^1\b\boI^i$ is
minimal, for $i=1,2$ and that $[\L,n,n-1,\b]$ is \emph{not} equivalent
to a simple stratum. Assume also (without loss of generality) that $n_1\ge n_2\ge 1$. 
Put $k_i=[\frac{n_i}2]+1$, for $i=1,2$, and
$\bk=(k_1,k_2)$. Then, by Lemma~\ref{ssdecomp}, $H^1=(\U^1(\L)\cap B)\U^\bk(\L)$ and $H^1\cap U=\U^\bk(\L)\cap U$. 
But, by definition of $\th$, it agrees with $\psi_\b$ on
$\U^\bk(\L)$ so the result follows.

\medskip

{\bf Degenerate semisimple case} (case (IV))
Now suppose $[\L,n,0,\b]$ is again a skew semisimple stratum, but with
$\b_2=0$. In this case it is straightforward to see that the flag
defining $U$ must be given by
$$
\{0\} \subset V_1=\la v_2\ra \subset V_2 = \la v_2,v_1\ra \subset
V_3=V_1\oplus V^1 \subset V_4=V,
$$
where $v_2\in V^2$ and $v_1\in V^1$. In particular, $U$ satisfies the
extra conditions in Corollary \ref{decomp}. Putting
$k_1=[\frac{n_1}2]+1$ and $k_2=1$, we have $H^1=(\U^1(\L)\cap
B)\U^\bk(\L)$ so, by Corollary \ref{decomp}, 
$$
H^1\cap U=(\U^1(\L)\cap B\cap U)(\U^\bk(\L)\cap U) = \U^\bk(\L)\cap U.
$$
As in semisimple case (i), the result now follows by the definition of
semisimple characters.

\medskip

{\bf Non-minimal semisimple case} (case (III))
Finally, suppose $[\L,n,0,\b]$ is a skew semisimple stratum, with splitting
$V=V^1\bigoplus V^2$, $1\le l\le 2$, that $\b_i=\boI^1\b\boI^i$ is
minimal, for $i=1,2$ but that $[\L,n,n-1,\b]$ \emph{is} equivalent
to a simple stratum. Let $r=-k_0(\b,\L)$. As in the simple case, we
will show that
\begin{itemize}
\item[($*$)] there exists a \emph{simple} stratum $[\L,n,r,\g]$ equivalent
to $[\L,n,r,\b]$ such that $\psi_\g|_{U_{der}}=1$.
\end{itemize}
(Indeed, $\g$ will be minimal.) The proof is then the same as that
in the simple case, since we can use the simple case for $\g$. In this
case we invoke Lemma~\ref{maxdecomp} to show that 
$H^1\cap U=H^{[\frac n2]+1}\cap U$.

As in the simple case, the flag corresponding to the unipotent subgroup $U$ is
given by $V_i=\la v,\b v,\cdots, \b^{i-1} v\ra$, for some $v\in V$
with $h(v,\b v)=0$. So $v=v_1+v_2$, with $v_i\in V^i$ such that
$h(v_1,\b_1 v_1)=-h(v_2,\b_2 v_2)$. 
Note that, for $i=1,2$, $\{v_i,\b_i v_i\}$ is a
basis for $V^i$, with respect to which $\b_i$ has matrix
$$
\begin{pmatrix} 0& \l_i\\ 1&0 \end{pmatrix},
$$
for some $\l_i\in F$. Also, let $E_i$ denote the linear map in
$A^{ii}$ which send $\b_i v_i$ to $v_i$ and $v_i$ to $0$; then
$E_i\in\fa_n^{ii}$. 

Let $[\L,n,r,\g_0]$ be a skew simple stratum equivalent to
$[\L,n,r,\b]$, with $\g_0\in M_{sp}$, and let $X^2-\l$ be the minimal
polynomial of $\g_0$ over $F$. For $i=1,2$, let $\g_i\in A^{ii}$ have
matrix
$$
\begin{pmatrix} 0& \l\\ 1&0 \end{pmatrix},
$$
with respect to the basis $\{v_i,\b_i v_i\}$ of $V^i$. Since
$\b-\g_0\in\fa_{-r}$, we get that $\l-\l_i\in\fa_{-n-r}^{ii}$, for 
$i=1,2$, so $(\l-\l_i)E_i\in\fa_{-r}^{ii}$. Hence $\b_i-\g_i\in
\fa_{-r}^{ii}$ and $\g=\g_1+\g_2$ is as required, since $\g v=\b v$.

\medskip

This completes the proof of
Proposition~\ref{thetaagreeswithpsi}.\hfill$\blacksquare$

\medskip

{\bf Remarks\/} It surely will not have escaped the reader's notice
that the methods in each case are rather similar. It may well be
possible to unify the cases into a single proof but we have not been
able to do this. We also note that we could not have used \cite{BH} Lemma~2.10 here, since the proof given there unfortunately does not work. It seems likely that the result there is true (at least in the tame case, as here) but we have not (yet) been able to find a proof.

%%%%%%%%%%%%%%%%%%%%%%%%%%%%%%%%%%%%%%%%%%%%%%%%%%%%%%%

\bigskip

\small

Corinne Blondel\hfill Shaun Stevens \break \indent
C.N.R.S. - Th{\'e}orie des Groupes - Case 7012%
\hfill School of Mathematics \break\indent
Institut de Math{\'e}matiques de Jussieu%
\hfill University of East Anglia \break\indent
Universit{\'e} Paris 7\hfill Norwich NR4 7TJ \break\indent
F-75251 Paris Cedex 05\hfill United Kingdom \break\indent
France

\smallskip
blondel@math.jussieu.fr\hfill ginnyshaun@bigfoot.com \break\indent

\end{document}